%% file: root_rev_20211112_ARXIV.tex
\documentclass[twocolumn]{autart}  

\usepackage{xcolor}

\usepackage{natbib}        
\usepackage{enumitem}

\input{Packages_DGA_Pack_4FPD}  

\input{NewCommands_DGA_Pack_4FPD_v2}

\makeatletter
\def\ps@copyright{%
 \let\@oddhead\@empty
 \let\@evenhead\@empty
 \def\@oddfoot{}%
 \let\@evenfoot\@oddfoot}
\makeatother

\begin{document}
\begin{frontmatter}

\title{On a probabilistic approach to synthesize control policies from example datasets\thanksref{footnoteinfo}\thanksref{footnoteinfo2}} 

\thanks[footnoteinfo]{{\bf This is an authors’ version of the work that is published in Automatica, Vol. 137,  2022. Changes were made to this version by the publisher prior to publication. The final version of record is available at https://doi.org/10.1016/j.automatica.2021.110121}}

\thanks[footnoteinfo2]{This publication has emanated from research conducted with the financial support of Science Foundation Ireland under Grant number 16/ RC/3872}

\author{Davide Gagliardi}$^{1}$, 
\author[Second]{Giovanni Russo}

\address[Second]{Department of Information and Electrical Engineering and Applied Mathematics, University of Salerno, Italy (e-mail: giovarusso@unisa.it)}
\thanks{Work done in part while at the School of Electrical and Electronic Engineering, University College Dublin, Ireland.}

\begin{abstract}
This paper is concerned with the design of control policies from example \textcolor{black}{datasets}. The case considered is when just a black box description of the system to be controlled is available and the system is affected by actuation constraints. These constraints are not necessarily fulfilled by the (possibly, noisy) example \textcolor{black}{data} and the system under control is not necessarily the same as the one from which \textcolor{black}{these data are} collected. In this context,  we introduce a number of theoretical  results to compute a control policy from \textcolor{black}{example datasets} that: (i) makes the behavior of the closed-loop system similar to the one illustrated in the \textcolor{black}{data}; (ii) guarantees compliance with the  constraints.  \textcolor{black}{We recast the control problem as a finite-horizon optimal control problem} and give an explicit expression for its optimal solution. \textcolor{black}{Moreover, } we turn our findings into an algorithmic procedure. The procedure gives a systematic tool to compute the policy. The effectiveness of our approach is illustrated via a numerical example, where we use real data collected from test drives to synthesize a control policy for the merging of a car on a highway.
\end{abstract}
 
\end{frontmatter}

\section{Introduction}

Model-based control is a key paradigm to design control systems: the design of platoons, fault tolerant and biochemical systems are just few of the frontier applications where this approach has been successfully used. Unfortunately, detailed mathematical models in the form of e.g. differential/difference equations, are not always available and, when available, can be hard to identify. Hence, a paradigm that is becoming increasingly popular is to synthesize control policies  directly from data, see e.g. \citep{8933093,8960476,HOU20133} and references therein. As noted in these papers, this approach aims at designing policies while bypassing the need to devise/identify a mathematical model and can be useful in applications where first-principle models cannot be obtained and/or system identification is too computationally expensive.  

An appealing framework to design controllers from data is that of using \textcolor{black}{example datasets}. This {\em learning from demonstrations} approach involves learning actions by observing an expert \citep{Han_lIU_zHU_Pas_19,8550288,10.1371/journal.pcbi.1007452}. In this context, we consider the key challenge of designing control policies from noisy example datasets for systems affected by actuation constraints.   Motivated by this, \textcolor{black}{we present\footnote{An early version of the results presented at the 21st IFAC World Congress \citep{Gagliardi_D_et_Russo_G_IFAC2020_extended_Arxiv}.} a set of technical results to synthesize policies directly from examples that might not satisfy the actuation constraints}.  \textcolor{black}{Situations where our results are of interest naturally arise in the context of  e.g. autonomous urban driving, where one is often interested in designing policies from driving examples, while ensuring that the control (e.g. speed, acceleration) signal fulfills certain properties with some acceptable/design probability level \citep{6760249}.} We now survey some related works on data-driven control and learning from \textcolor{black}{examples}.

\noindent {\bf Data-driven control.} As noted in e.g. \citep{8933093,8960476}, work on data-driven control can be traced back to \citep{10.1115/1.2899060} and their results on the tuning of PID controllers. Recently, driven by the explosion in the {\em amount} of available data, the problem of finding control policies from datasets has gained increasing attention, see e.g. \citep{Hou_09}. For example,  \textcolor{black}{\citep{8960476} studies data-driven control in LTI systems when the data available are not persistently exciting and where, therefore, it is impossible to perform subspace identification.} Recent works also include \citep{TANASKOVIC20171}, where a direct data-driven design approach is introduced for discrete-time stabilizable systems with Lipschitz nonlinearities and e.g.  \citep{7068299,8453019,8703172} that consider the problem of designing controllers for systems that have an underlying linear dynamics. \textcolor{black}{We also recall here \citep{8933093} that derives a parametrization of linear feedback systems using data-dependent linear matrix inequalities.} Besides works on e.g. data-based tuning of PID controllers \citep{4610025} other remarkable results have been obtained by taking inspiration from the rich literature on Model Predictive Control (MPC). These include \citep{8039204}, where an MPC learning algorithm is introduced for iterative tasks when the system dynamics is partially known, \citep{SALVADOR2018356} where a data-based predictive control algorithm is presented for unknown time-invariant systems and \textcolor{black}{\citep{8795639,9028943}} that, by taking a behavioral systems perspective, introduce a data-enabled predictive control algorithm for data generated by LTI systems. 


\noindent {\bf Learning from \textcolor{black}{example datasets}.} At their roots, { learning from demonstration techniques} largely rely on inverse optimal control \citep{506395}. Nowadays, these techniques\textcolor{black}{, which essentially aim at finding/improving a policy from an initial demonstration dataset \citep{9317713},} are recognized as a convenient framework to learn parametrized policies from {\em success stories} \citep{ARGALL2009469} and potential applications include planning \citep{doi:10.1177/0278364917745980} and medical prescriptions learning \citep{8796280}. There is then no surprise that, over the years, a number of techniques have been developed to tackle the  problem of learning parametrized control policies from demonstrations, mainly in the context of Markov Decision Processes. Results include \textcolor{black}{\citep{Han_lIU_zHU_Pas_19}, where parametrized policies consistent with data for a Markov decision process are learned via regularized logistic regression and regret bounds are given,} \citep{Ratliff2009}, which leverages a linear programming approach, \citep{Ratliff:2006:MMP:1143844.1143936} which relies on a maximum margin approach, \citep{ziebart2008maximum} that makes use of the maximum entropy principle, \citep{Ramachandran:2007:BIR:1625275.1625692} that formalizes the problem via Bayesian statistics and  \textcolor{black}{ \citep{10.1145/1015330.1015430}, where rewards are learned by expressing them as a combination of known features}.  \textcolor{black}{We also recall the recent \citep{pmlr-v97-edwards19a}, where an approach based on characterizing causal effects of latent actions is proposed to tackle imitation learning problems.} A Bayesian approach \citep{Peterka_V_Bayesian_Approach_to_sys_ident_1981} to dynamical systems is also at the basis of works such as \citep{Karny_M_Automatica_1996_Towards_Fully_Prob,Karny_M+Guy_TV_Sys&Ctr_lett_2006,Herzallah_R_JNeurNet_2015,Pegueroles_G+Russo_G_ECC19_confid,10.1007/978-3-030-01713-2_20,KARNY2012105} \textcolor{black}{see also \citep{Kappen_2012,Todorov11478,NIPS2006_d806ca13} and references therein}, which formalize \textcolor{black}{(via the so-called Kullback-Leibler divergence, see Section \ref{sec:KL_divergence} for the definition)} the control problem as the problem of minimizing a cost that captures the discrepancy between an ideal probability density function and the actual probability density function of the system under control.  
\textcolor{black}{Further, we recall \citep{6716965}, which also consider these Kullback-Leibler control problems without constraints.  In  such a work,   by leveraging an average cost formulation, the probability mass function for the state transitions is found. } {Finally, we also recall  \citep{9244209,9446558}, where policies are obtained from the minimization of similar costs by leveraging multiple, specialized,  data sources.} 

\subsection*{Contributions of this paper}
We introduce a framework to design policies from example datasets for systems having actuation constraints and for which just a probabilistic description is known. \textcolor{black}{With our results, we aim at synthesizing control policies that make the behavior of the closed loop system {\em similar} (in a sense  defined in Section \ref{sec:control_problem}) to the one seen in the example \textcolor{black}{data}, while still fulfilling the system-specific actuation constraints. To the best of our knowledge, the formulation presented here is novel and these are the first results that, simultaneously: (i) do not require assumptions on the linearity of the underlying stochastic dynamics that is generating the data; (ii) allow for the example \textcolor{black}{data} to be noisy and collected  from a system that is not necessarily the same as the one that is under control; (iii) do not need the actuation constraints  to be necessarily fulfilled in the example \textcolor{black}{data}; (iv) do not need an a-priori parametrization of the policy.} Our key \textcolor{black}{technical} contributions are summarized as follows: 
\begin{itemize}
\item \textcolor{black}{we recast the control problem as the problem of {\em reshaping} certain probability density functions that can be obtained from data.  This leads to formulate a constrained infinite dimensional, finite-horizon, optimal control problem.  The decision variables of the  problem are probability densities  and the constraints are linear functionals in these densities}; 
\item we introduce a number of theoretical results to tackle the control problem.  \textcolor{black}{Specifically, we prove that the finite-horizon optimal control problem can be split into infinite-dimensional sub-problems, with each of these sub-problems  having a probability density (i.e. the control policy) as decision variable.  After showing that each of the sub-problems is convex,  we solve them explicitly.  By leveraging this, we show that the control problem can be solved iteratively via a backward recursion. This leads to an explicit expression of the optimal policy};
\item the results are turned into an algorithmic procedure \textcolor{black}{to compute  the  optimal control policy};
\item  we illustrate the effectiveness of the results via a numerical example that involves the use of real data.
\end{itemize}
\textcolor{black}{The paper is organized as follows. After introducing some mathematical preliminaries in Section \ref{ss:Notat}, we formalize the control problem (Problem \ref{prob:Main_Constr_Ctrl}, in Section \ref{ss:Assumpt}). This problem is solved with the technical results presented in Section \ref{sec:results} (some of the proofs are in the appendix).  Specifically, in this section we first introduce, and discuss, an auxiliary problem (i.e. Problem \ref{prob:Lemma_1}). With Lemma \ref{lem:Constrained_KL} we then give an explicit expression for the optimal solution of Problem \ref{prob:Lemma_1} and this result is used iteratively to prove Theorem \ref{theo:ctrlConstr}, which gives the optimal solution to Problem \ref{prob:Main_Constr_Ctrl}.  We conclude Section \ref{sec:results} by highlighting a connection between our approach and maximum entropy. The results are turned into an algorithmic procedure (Section \ref{sec:algorithm}) and a numerical example is presented in Section \ref{sec:example}. Concluding remarks are given in Section \ref{sec:conclusions}. }
\section{Mathematical Preliminaries} \label{ss:Notat}

Sets, as well as operators, are denoted by {\em calligraphic} characters, while vector quantities are denoted in {\bf bold}.   \textcolor{black}{Random (row) vectors (i.e. a multidimensional random variables) are denoted by upper-case bold letters and their realization is denoted by lower-case bold letters.} For example, $\mathbf{Z}$ denotes a multi-dimensional random variable and its realization is denoted by $\mathbf{z}$. The \textit{probability density function} (or simply \textit{pdf} in what follows) of a continuous $\mathbf{Z}$ is denoted by $\textcolor{black}{f(\mathbf{z})}$.  The support of $\textcolor{black}{f(\mathbf{z})}$ is denoted by $\mathcal{S}\left(\textcolor{black}{f}\right)$ and, analogously, the  expectation of a function $\mathbf{h}(\cdot)$ of $\mathbf{Z}$ is indicated with $\E_{\textcolor{black}{f}}[\mathbf{h}(\mathbf{Z})]$ 
and defined as $\E_{{{f}}}[\mathbf{h}(\mathbf{Z})] := \int_{\mathcal{S}\left(\textcolor{black}{f}\right)}\mathbf{h}(\mathbf{z})f(\mathbf{z})d\mathbf{z}$.  For notational convenience,  {whenever it is clear from the context, we omit the  domain of integration in the integral}. We also remark here that: (i) whenever we apply the averaging operator to a given function, we use an upper-case letter for the function argument as this is a random vector; (ii) to stress the linearity of certain functionals or operators with respect to a specific argument, we include that argument in square brackets. 
The \textit{joint} pdf of two \textcolor{black}{(non independent)} random vectors, say $\mathbf{Z}$ and $\mathbf{Y}$, is denoted by  $f(\mathbf{z,y})$.
The \textit{conditional} probability density function (or {\em cpdf} in what follows) of $\mathbf{Z}$ with respect to $\mathbf{Y}$ is denoted by $f\left( \mathbf{z}| \mathbf{y} \right)$ and, whenever the context is clear, we use the shorthand notation $\tilde{f}_\mathbf{Z}$. Finally, given $\mathcal{Z}\subseteq\R^{\textcolor{black}{n}_z}$, its \textit{indicator function} is denoted by $\mathds{1}_{\mathcal{Z}}(\mathbf{z})$. That is, $\mathds{1}_{\mathcal{Z}}(\mathbf{z}) =1$, $\forall \mathbf{z} \in \mathcal{Z}$ and $0$ otherwise. 
We also make use of the internal product between tensors, which is denoted by $\langle  \cdot, \cdot \rangle$, while $\mathcal{A}\setminus\mathcal{B}$ is the set difference between $\mathcal{A}$ and $\mathcal{B}$. 
We indicate (ordered) countable sets as $\lbrace w_k \rbrace_{\mathcal{K}}:= \lbrace w_k \rbrace_{k = k_1}^{k_n}$, where $w_k$ is the generic element belonging to the set and $k_1$, $k_n$ are the indices of the first and last element respectively.  

\subsection{The Kullback-Leibler divergence}\label{sec:KL_divergence}
The control problem \textcolor{black}{is} stated (see Section \ref{sec:control_problem}) in terms of the Kullback-Leibler (KL, \citep{KL_51}) divergence, formalized with the following:

\begin{Def}\label{def:KLdiv}	
Consider two pdfs, \textcolor{black}{$\phi(\mathbf{z})$} and \textcolor{black}{$g(\mathbf{z})$,} with the former being absolutely continuous with respect to the latter. Then, the \KL-divergence of $\phi(\mathbf{z})$ with respect to $g(\mathbf{z})$, is $	\mathcal{D}_{\KL}
	\left(\phi(\mathbf{z}) ||g(\mathbf{z}) \right):
	= \int_{\mathcal{S}(\phi)} \phi(\mathbf{z}) \; \ln\left( \frac{\phi(\mathbf{z})}{g(\mathbf{z})}\right)\,d\mathbf{z}$.
\end{Def}
Intuitively,  \textcolor{black}{$\mathcal{D}_{\KL}\left(\phi(\mathbf{z}) ||g(\mathbf{z}) \right)$} is a measure of the proximity of the pair of pdfs,  \textcolor{black}{$\phi(\mathbf{z})$} and \textcolor{black}{$g(\mathbf{z})$}. We now give a property of the KL-divergence\textcolor{black}{, also known as {\em chain rule},} used in the proof of Theorem \ref{theo:ctrlConstr}.
\begin{Property}\label{proper:KLsplit} Let $\mathbf{Z}$ and $\mathbf{Y}$ be two random vectors and let \textcolor{black}{$\phi(\mathbf{y},\mathbf{z})$} and \textcolor{black}{$g(\mathbf{y},\mathbf{z})$} be two joint pdfs.  Then, the following identity holds:
\textcolor{black}{$
		\mathcal{D}_{\KL}	
		\left(
		\phi (\mathbf{y},\mathbf{z}) || \, g (\mathbf{y},\mathbf{z}) 
		\right) 	=		\mathcal{D}_{\KL}
		\left( \phi (\mathbf{y}) || \, g (\mathbf{y} ) \right) + 
		\mathbb{E}_{ \phi  }
		\left[	
		\mathcal{D}_{\KL} \left( 	\phi(\mathbf{z}|\mathbf{Y}) || \, g(\mathbf{z}|\mathbf{Y}) \right)
		\right]$.}

\proof 
\textcolor{black}{ the chain rule can be found in e.g. \citep{10.5555/1146355} and its} proof follows from the definition of $\mathcal{D}_{\KL}$, the \textit{conditioning} and \textit{independence} rules for pdfs.  \textcolor{black}{We give a self-contained} proof in the appendix. 
\qed 	
\end{Property}

\section{Formulation of the Control Problem}
\label{ss:Assumpt}

Let: (i) $\mathcal{K} :=\lbrace k \rbrace_{k=1}^n$, $\mathcal{K}_0 := \mathcal{K} \cup \lbrace 0 \rbrace$ and $\mathcal{T}:= \lbrace t_k:  k \in\mathcal{K}_0\rbrace$; (ii) $\mathbf{x}_k\in\R^{\textcolor{black}{n}_x}$  and $\mathbf{u}_k\in\R^{\textcolor{black}{n}_u}$  be, respectively, the system state and input at time $t_k\in\mathcal{T}$; (ii) $\textcolor{black}{\boldsymbol{\Delta}}_k:=(\mathbf{x}_k,\mathbf{u}_k)$ be the \textcolor{black}{dataset} collected from the system at time $t_k\in\mathcal{T}$ and  $\textcolor{black}{\boldsymbol{\Delta}}^k$ the \textcolor{black}{dataset} collected from $t_0\in\mathcal{T}$ up to time $t_k\in\mathcal{T}$ ($t_k>t_0$). See also \citep{Peterka_V_Bayesian_Approach_to_sys_ident_1981}, where it is  shown that the system behavior can be described via the joint pdf of the observed \textcolor{black}{dataset}, say $f(\textcolor{black}{\boldsymbol{\Delta}}^n)$. Moreover, \textcolor{black}{by making the standard assumption that Markov's property holds \citep{Karny_M_Automatica_1996_Towards_Fully_Prob},} the chain rule for pdfs leads to the following factorization for $f(\textcolor{black}{\boldsymbol{\Delta}}^n)$:
\begin{equation}
\label{eq:CL_red_v1}
f\left( \textcolor{black}{\boldsymbol{\Delta}}^n \right)  
=
\prod_{k\in \mathcal{K}} 
f	\left(	\mathbf{x}_k | \mathbf{u}_k, \mathbf{x}_{k-1} \right)
f	\left( \mathbf{u}_k | \mathbf{x}_{k-1} \right) 
f 	\left(	\mathbf{x}_0 	\right).
\end{equation}
Throughout this work we refer to (\ref{eq:CL_red_v1}) as the \textit{probabilistic description of the closed loop system}, or we simply say that (\ref{eq:CL_red_v1}) is our\textit{ closed loop} system. 

\begin{Rmk}
The cpdf $f\left(\mathbf{x}_k | \mathbf{u}_k, \mathbf{x}_{k-1} \right)$  describes the system behavior at time $t_k$, given the previous state and the input at time $t_k$.  The input is generated by the \textcolor{black}{cpdf} $f	\left( \mathbf{u}_k | \mathbf{x}_{k-1} \right)$. This is a {\em randomized control policy} returning the input given the previous state. Note that the initial conditions are embedded in the probabilistic system description through the prior $f \left(	\mathbf{x}_0 \right)$.
\end{Rmk}

In the rest of the paper we use the following \textit{shorthand} notations: $\tilde f_{\mathbf{X}}^k:=f	\left(	\mathbf{x}_k | \mathbf{u}_k, \mathbf{x}_{k-1} \right)$, 
$\tilde f_{\mathbf{U}}^k:=f	\left( \mathbf{u}_k | \mathbf{x}_{k-1} \right)$, $\tilde f^k:=\tilde f_{\mathbf{X}}^k\tilde f_{\mathbf{U}}^k$, $f_0:=f 	\left(	\mathbf{x}_0\right)$ and $f^n:= f\left( \textcolor{black}{\boldsymbol{\Delta}}^n \right)$. Hence, (\ref{eq:CL_red_v1}) can be compactly written as
\textcolor{black}{\begin{equation}
\label{eq:CL_red_v1_short}
f^n
=
\left(\prod_{k\in \mathcal{K}} \tilde f_{\mathbf{X}}^k\tilde f_{\mathbf{U}}^k\right)f_0 .
\end{equation}}

\subsection{The control problem}\label{sec:control_problem}
Our goal is to synthesize, from an example dataset, say  $\textcolor{black}{\boldsymbol{\Delta}}_e^n $, the control pdf/policy, \textcolor{black}{i.e. $\tilde f_{\mathbf{U}}^k$}, that: (i) makes the closed loop system {\em similar} to the behavior illustrated in the \textcolor{black}{dataset}; (ii)  satisfies the system actuation constraints even if these are not fulfilled by the examples. We specify the behavior illustrated in the examples through the reference pdf $\Ipdf \left( \textcolor{black}{\boldsymbol{\Delta}}_e^n \right) $ extracted from the example dataset. \textcolor{black}{By means of Markov's property} and the chain rule for pdfs we have $\Ipdf \left( \textcolor{black}{\boldsymbol{\Delta}}_e^n \right)  
	:=
	\prod_{k\in \mathcal{K}} 
	\Ipdf\left(	\mathbf{x}_k | \mathbf{u}_k, \mathbf{x}_{k-1} \right)
	\Ipdf\left( \mathbf{u}_k | \mathbf{x}_{k-1} \right) 
	\Ipdf\left(	\mathbf{x}_0\right)$. Again, by setting 
	$\tilde g_{\mathbf{X}}^k:=g	\left(	\mathbf{x}_k | \mathbf{u}_k, \mathbf{x}_{k-1} \right)$, 
	$\tilde g_{\mathbf{U}}^k:=g	\left( \mathbf{u}_k | \mathbf{x}_{k-1} \right)$, $\tilde g^k:=\tilde g_{\mathbf{X}}^k\tilde g_{\mathbf{U}}^k$, $g_0:=g 	\left(	\mathbf{x}_0\right)$
 and $g^n:=g\left( \textcolor{black}{\boldsymbol{\Delta}}_e^n \right)$ we get:
 \textcolor{black}{\begin{equation}
\label{eq:CL_red_ref_short}
g^n
=
\left(\prod_{k\in \mathcal{K}} \tilde g_{\mathbf{X}}^k\tilde g_{\mathbf{U}}^k\right)g_0 .
\end{equation}}
\textcolor{black}{The pdf $\tilde g_{\mathbf{X}}^k$ is not necessarily the same as the pdf $\tilde f_{\mathbf{X}}^k$. This allows to consider situations of practical interest where e.g. the system used to collect the example dataset is not necessarily the same as the system under control\footnote{\textcolor{black}{As noted in \citep{Peterka_V_Bayesian_Approach_to_sys_ident_1981} the probabilistic descriptions are the {\em most general description of a system from the viewpoint of an outer observer}. These pdfs can be computed from the available data as e.g. empirical distributions}}.}
\textcolor{black}{\begin{Rmk}\label{rem:dataset}
Within this paper, a dataset is a sequence of data.  In certain applications one might have access to a collection of datasets,  which can be leveraged to compute the above pdfs.  In e.g.  autonomous driving applications one might run multiple test drives and, in this case, the pdfs can be computed from the collection of the speed profiles (i.e. the collection of datasets) obtained from each test drive. This observation is exploited in Section \ref{sec:example}.
\end{Rmk}}

The control problem, formalized next, can then be recast as the problem of designing \textcolor{black}{$\tilde f_{\mathbf{U}}^k$} so that $f^n$ \textcolor{black}{(i.e. the probabilistic description of the closed loop system)} approximates $g^n$ \textcolor{black}{(i.e.  the joint pdf extracted from the example dataset)}. This is formally stated as follows: 
\begin{Prob}
\label{prob:Main_Constr_Ctrl}
{Let, $ \forall k\in\mathcal{K}$:
\begin{enumerate}
	\item[(i)]  $n_e^k$ and $n_l^k$ be  positive integers;
	\item[(ii)]  
	\textcolor{black}{$h_{\mathbf{u},j}^k : \mathcal{U}_k \subseteq \R^{\textcolor{black}{n}_u} \mapsto\R$}, $ j = 1,\dots, n_e^k + n_l^k$,	be measurable mappings from $\mathcal{U}_k \subseteq \R^{\textcolor{black}{n}_u}$ into $\mathbb{R}$;  
		\item[(iii)] $H_{\mathbf{u},j}^k\in \mathbb{R}$, $ j = 1,\dots, n_e^k + n_l^k$, be  constants;
	\item[(iv)]
{	$h_{\mathbf{u},0}^k (\mathbf{u}) :=  \mathds{1}_{\mathcal{U}_k } (\mathbf{u})$ and $H_{\mathbf{u},0}^k	:= 1$.}
\end{enumerate}	
}	
Find  {$\left \lbrace 
\left( \tilde{f}_\mathbf{U}^k \right)^\ast
\right\rbrace_{\mathcal{K}}:=\left \lbrace f^\ast \left( \mathbf{u}_k | \mathbf{x}_{k-1} \right)\right\rbrace_{\mathcal{K}}$} such that:
\begin{equation}
\label{eq:ProbMain}	
	\begin{array}{lcll} 
	{ \left \lbrace 
	\left( \tilde{f}_\mathbf{U}^k \right)^* 
	\right \rbrace_{\mathcal{K}} \in \arg } 
	& \underset{{\left \lbrace\tilde{f}_\mathbf{U}^k\right \rbrace_{\mathcal{K}}}}{\min} 
		& \mathcal{D}_{\KL}\left( f^n||\Ipdf^n \right)\\
	& { \textnormal{s.t.:} }  	
		& c_{\mathbf{u},j}^k \left[ \tilde{f}_\mathbf{U}^k  \right]= 0, 	
			& {\forall j \in\mathcal{E}_0^k,  k \in \mathcal{K} }, \\
	&	
		& c_{\mathbf{u},j}^k \left[ \tilde{f}_\mathbf{U}^k  \right]\leq 0,  
			& { \forall j \in\mathcal{I}^k,    k \in \mathcal{K} } \\	
	\end{array} 
\end{equation}	
where 
\begin{equation}
\label{eq:constr_c_j_k}
c_{\mathbf{u},j}^k \left[ \tilde{f}_\mathbf{U}^k  \right]:=
\E_{\tilde{f}_\mathbf{U}^k }
\left[
h^{k}_{\mathbf{u},j}
\left(\mathbf{U_k}\right)
\right] 
- H^k_{\mathbf{u},j},
\end{equation}		
and  where $ \mathcal{E}^k_0:= \mathcal{E}^k\cup \lbrace 0 \rbrace$
{(with $\mathcal{E}^k:= \lbrace j \rbrace^{n^k_e}_{1}$)}, $\mathcal{I}^k:= \lbrace j \rbrace^{n^k_e + n^k_l}_{n^k_e + 1}$ 
are the \textit{equality} and \textit{inequality} {constraints} index sets
at time $t_k$ respectively.
\end{Prob}
\textcolor{black}{Following Definition \ref{def:KLdiv},  the cost in (\ref{eq:ProbMain}) is well defined if $f^n$ is absolutely continuous with respect to {$g^n$}. That is, as for other control results based on the minimization of the KL-divergence, the cost is well defined if the former pdf is $0$ whenever the latter pdf is $0$}.
\textcolor{black}{The constraints in Problem \ref{prob:Main_Constr_Ctrl}  are specific to the system under control and we do not make any assumption on the fact that the behavior observed in the example dataset fulfills these constraints. We make the following remarks:
\begin{Rmk}
Since the pdf $ \tilde g_{\mathbf{X}}^k$ can be different from $\tilde f_{\mathbf{X}}^k$,  we have that  $\tilde g_{\mathbf{U}}^k$ is not, in general, the optimal solution to Problem \ref{prob:Main_Constr_Ctrl}. Further,  since $\tilde g_{\mathbf{U}}^k$ might not satisfy the constraints in Problem \ref{prob:Main_Constr_Ctrl},  this pdf might not  be feasible for the problem. The explicit expression for the optimal solution to Problem \ref{prob:Main_Constr_Ctrl} is instead given in Theorem \ref{theo:ctrlConstr} and, as we shall see, it depends on $\tilde g_{\mathbf{U}}^k$. 
\end{Rmk}
\begin{Rmk}
The control problem is cast as the problem of
designing $\tilde{f}^k_{\mathbf{U}}$ so that $\mathcal{D}_{\KL}\left(f^n||g^n\right)$ is minimized, subject to the constraints.  In this sense,  the policy is designed so that $f^n$ {\em approximates} $g^n$ while still fulfilling the constrains.  With Theorem \ref{theo:ctrlConstr} we give an explicit expression for the optimal solution $\left \lbrace 
	\left( \tilde{f}_\mathbf{U}^k \right)^\ast 
	\right \rbrace_{\mathcal{K}}$.  The control input,  say $\bf{u}_k$,  applied to the system at time-step $k$ is obtained by sampling from the pdf $\left( \tilde{f}_\mathbf{U}^k \right)^\ast$.
	\end{Rmk}}

As noted in \citep{Pegueroles_G+Russo_G_ECC19_confid}, in the special case when in Problem \ref{prob:Main_Constr_Ctrl} there are no constraints and the pdfs $f^n$, $g^n$ are normal distributions with zero mean, then the control policy solving the problem has the same update rules as the Linear Quadratic Regulator. The introduction of the constraints formalized in Problem \ref{prob:Main_Constr_Ctrl} \textcolor{black}{can be useful in situations of practical interest where the actuation capabilities of the system are different from the actuation capabilities of the system used to collect the example \textcolor{black}{data} and/or in situations where the example data are collected without necessarily knowing the specific constraints of the system under control.}Also, by embedding actuation constraints into the problem formulation and by solving the resulting problem, one can {\em export} the policy (by e.g. using it to generate example datasets) that has been synthesized on a given system to other systems having different actuation capabilities. 

\textcolor{black}{As we shall see, the solution to Problem \ref{prob:Main_Constr_Ctrl} depends on the  conditional pdf \textcolor{black}{$\tilde f_{\mathbf{X}}^k$}, i.e.  the cpdf describing the behavior of the system under control, which therefore needs to be properly estimated. Hence,  as for other data-driven control approaches that rely solely on the available data,  these need to be sufficiently {\em informative}.  We refer  to e.g. \citep{8960476,8933093,9062331,COLIN2020109000} for recent results on data informativity and to \citep{9406124,10.2307/2238570} for results on optimal experimental design and persistency of excitation.}

\begin{Rmk}\label{rem:constraints_exp}
In Problem \ref{prob:Main_Constr_Ctrl}, the constraints are formalized as expectations and can be equivalently written as $\int_{\mathcal{S}(\tilde{f}_\mathbf{U}^k )} \tilde{f}_\mathbf{U}^k \; h_{\mathbf{u},j}^k\left( \mathbf{u} \right)\; d\mathbf{u} =  H_{\mathbf{u},j}^k$. The equality and inequality constraints\textcolor{black}{, and their number,} can change over time. \textcolor{black}{In Problem \ref{prob:Main_Constr_Ctrl}, $n_e^k$ and $n_l^k$ denote, respectively, the number of equality and inequality constraints at time-step $k$ (see the definitions of the sets $\mathcal{E}^k$ and $\mathcal{I}^k$ in the problem statement).}{Finally, the first equality constraint is a normalization constraint on the solution of the problem.}
\end{Rmk}
\textcolor{black}{\begin{Rmk}\label{rem:constraints_new}
The constraints in \eqref{eq:ProbMain} can be used to guarantee properties on e.g.  the moments of the cdpfs $\left( \tilde{f}_\mathbf{U}^k \right)^* $, i.e. on the optimal solution of the problem (see also Remark \ref{rem:constraints} and Remark \ref{rem:bound_constr}).  Additionally,  we  note that the inequality constraints in (\ref{eq:ProbMain}) can be used to capture {\em bound constraints} of the form $\mathbb{P}(\mathbf{U_k}\in\bar{\mathcal{U}}_k)\ge 1-\varepsilon$, where $\bar{\mathcal{U}}_k \subset \mathcal{U}_k$ and $\varepsilon\ge 0$. Indeed, note that: $\mathbb{P}(\mathbf{U_k}\in\bar{\mathcal{U}}_k)  = \int_{\mathcal{S}(\tilde{f}_\mathbf{U}^k )} \mathds{1}_{\bar{\mathcal{U}}_k}(\mathbf{u}_k)\tilde{f}_\mathbf{U}^k d\mathbf{u}_k
 = \E_{\tilde{f}_\mathbf{U}^k}\left[\mathds{1}_{\bar{\mathcal{U}}_k}(\mathbf{U}_k)\right]$. Hence,  the constraint $\mathbb{P}(\mathbf{U_k}\in\bar{\mathcal{U}}_k)\ge 1-\varepsilon$ can be written as a constraint of the form of these in (\ref{eq:ProbMain}).  Typically, in  works on {\em safe learning} for stochastic systems, see e.g. \citep{8651519,9290355}, $\varepsilon$ is a small constant and hence these types of constraints model the fact that the probability that the control variable is outside some (e.g. desired) set $\bar{\mathcal{U}}_k$ is smaller than some {\em acceptable} $\varepsilon$. When $\varepsilon =0$, the constraint amounts at imposing that $\mathbb{P}(\mathbf{U_k}\in\bar{\mathcal{U}}_k)  = 1$, thus implying that the pdf $\tilde{f}_\mathbf{U}^k$ is zero outside the set $\bar{\mathcal{U}}_k$.
\end{Rmk}
\begin{Rmk} 
In relation to Remark \ref{rem:constraints_new}, we also note how the constraint $\mathbb{P}(\mathbf{U_k}\in\bar{\mathcal{U}}_k)=\E_{\tilde{f}_\mathbf{U}^k}\left[\mathds{1}_{\bar{\mathcal{U}}_k}(\mathbf{U}_k)\right]\ge 1-\varepsilon$ is linear, and hence convex, in the decision variable \textcolor{black}{(i.e.  $\tilde{f}_\mathbf{U}^k$)} even if the indicator function is not convex, see e.g. \citep{6760249}. This implies that, in Problem \ref{prob:Main_Constr_Ctrl}, these constraints can be handled without resorting to {bound} approximations. These approximations are  typically used in the literature to handle the intrinsic non-convexity of the constraint arising in problems where the decision variable is $\mathbf{U}_k$, see e.g. \citep{6760249,8287310}. Note indeed that,  while $\E_{\tilde{f}_\mathbf{U}^k}\left[\mathds{1}_{\bar{\mathcal{U}}_k}(\mathbf{U}_k)\right]$ is convex in $\tilde{f}_\mathbf{U}^k$, this is not convex in $\mathbf{U}_k$. 
\end{Rmk}}
\section{Technical Results}\label{sec:results}
We now introduce our main results. The key result behind the algorithm of  Section \ref{sec:algorithm} is Theorem \ref{theo:ctrlConstr}. The proof of this theorem, given in this section, makes use of two technical results (i.e. Lemma \ref{lem:Constrained_KL} and Lemma \ref{proper:KLdiv_split_nD_fn}). With our first result, i.e. Lemma \ref{lem:Constrained_KL}, we tackle an {\em auxiliary} problem that is iteratively solved within the proof of Theorem \ref{theo:ctrlConstr}. This auxiliary problem is formalized next.

\begin{Prob}\label{prob:Lemma_1}
Let: 
\begin{enumerate}
\item[(i)] $n_e$, $n_l$ be two positive integers;
\item[(ii)] $\mathbf{Z}$ be a random vector, with $\mathbf{z}\in\mathcal{Z}\subseteq \R^{n_z}$;
\item[(iii)] \textcolor{black}{$g (\mathbf{z})$} and \textcolor{black}{$ f(\mathbf{z})$}  be two pdfs having support $\mathcal{Z}$; 
\item[(iv)] $\alpha: \mathcal{Z} \mapsto \mathbb{R}$ be a mapping from $\mathcal{Z}$ into $\mathbb{R}$, which is integrable under the measure given by $ f(\mathbf{z})$;
{\item[(v)] \textcolor{black}{
$h_j:\mathcal{Z}\mapsto\R$}, $j=1,\ldots,n_e+n_l$, be  measurable  mappings from $\mathcal{Z}$ into $\mathbb{R}$; 
\item[(vi)] $H_j \in \mathbb{R}$, $j=1,\ldots,n_e+n_l$, be a constant;
\item[(vii)] 	
$h_0 (\mathbf{z}) :=  \mathds{1}_{\mathcal{Z} } (\mathbf{z})$ and $H_0:= 1$.
}
\end{enumerate}
Find the pdf \textcolor{black}{$f^{\ast}(\mathbf{z})$} such that:
\begin{equation}
\label{eq:probl_Lemma_1}
	\begin{array}{lcll} 
	{ \textcolor{black}{f^{\ast}(\mathbf{z})} \in \arg }
		& \underset{\textcolor{black}{f(\mathbf{z})}}{ \min} 	
			& \mathcal{L}\left(\textcolor{black}{f(\mathbf{z})}\right) \\	
		& {\textnormal{s.t.:}} 			
			&  c_j \left[\textcolor{black}{f(\mathbf{z})}\right]= 0,   		
				& {\forall  j \in \mathcal{E}_0}, \\
		& 	& c_j \left[\textcolor{black}{f(\mathbf{z})}\right]\leq 0,  	
				&{\forall  j \in \mathcal{I}}
	\end{array} 
\end{equation}
where:
\begin{equation}
\label{eq:Lagr_Lem1}
\mathcal{L}\left(\textcolor{black}{f(\mathbf{z})}\right) := \mathcal{D}_{\KL} \left(\textcolor{black}{f(\mathbf{z})} || \textcolor{black}{g(\mathbf{z})}\right)	+ 
\int f (\mathbf{z}) \; \alpha \left( \mathbf{z} \right) \, d\mathbf{z},
\end{equation}	
with
\begin{equation}
\label{eq:constr_c_j}
c_j \left[\textcolor{black}{f(\mathbf{z})}\right]:=
\int{f \left( \mathbf{z} \right) \;	
h_j \left( \mathbf{z} \right)\; d\mathbf{z}} 
	- H_j,
\end{equation}	
{
and
$ \mathcal{E}_0:= \mathcal{E}\cup \lbrace 0 \rbrace$, $\mathcal{E}:= \lbrace j \rbrace^{n_e}_{1}$, $\mathcal{I}:= \lbrace j \rbrace^{n_e + n_l}_{n_e + 1}$. 
}
\end{Prob} 
\textcolor{black}{In Lemma \ref{lem:Constrained_KL} we  give the optimal solution of Problem \ref{prob:Lemma_1}. Such a  result is  used in  Theorem \ref{theo:ctrlConstr} to find the solution to Problem \ref{prob:Main_Constr_Ctrl}.  Consider the special case where, in Problem \ref{prob:Main_Constr_Ctrl}: (i) the function $\alpha(\cdot)$ is constant in $\mathcal{Z}$; (ii) the support $\mathcal{Z}$ is compact and $g(\mathbf{z})$ is uniform. Then,  the cost in Problem \ref{prob:Lemma_1} becomes $\mathcal{L}(f(\mathbf{z})) = \int f(\mathbf{z})\log f(\mathbf{z})d\mathbf{z}$. Hence,  in this case, Problem \ref{prob:Lemma_1} becomes a constrained entropy maximization problem.   See also Remark \ref{rem:maxent_maxlikelihood_comaprison}. }
{We consider feasible sets of constraints that satisfy the following constraint qualification condition (CQC):
\begin{Def}
Consider the set of linear constraints 
$$
\left\{\begin{array}{*{20}l}
c_i\left[\textcolor{black}{f(\mathbf{z})}\right] =0, & i=1,\ldots,n_e\\
c_j\left[\textcolor{black}{f(\mathbf{z})}\right] \le 0, & j = 1,\ldots, n_l.
\end{array}\right.
$$
The Slater's CQC (or simply Slater's condition in what follows) is said to hold for such a set if there exists a pdf, say $\bar{f}(\mathbf{z})$, such that $c_i[\bar{f}(\mathbf{z})] =0$, $\forall i\in\{1,\ldots,n_e\}$ and $c_j[\bar{f}(\mathbf{z})] < 0$, $\forall j \in \{ 1,\ldots, n_l\}$.
\end{Def}
\begin{Rmk}\label{rem:duality}
Slater's condition is also known in the literature on infinite-dimensional optimization problems as a {\em regularity} condition on the constraint set \citep{Roc_76,Roc_88}. For a convex infinite-dimensional optimization problem, the fulfillment of such a condition guarantees strong duality \citep{Roc_88}.
\end{Rmk}
We make the following:
\begin{Asm}\label{asm:constrraints}
The constraints sets in (\ref{eq:ProbMain}) and in (\ref{eq:probl_Lemma_1}) are feasible and satisfy Slater's condition.
\end{Asm}
\begin{Rmk}
Assumption \ref{asm:constrraints} is widely used in the literature on e.g. infinite-dimensional optimization \citep{Roc_88,Fan_K_JMAnnApp_1968_infinite}, divergences optimization and cross-entropy problems \citep{10.1145/2591796.2591803,Bot_05}. If the problems involve discrete distributions, then checking feasibility and the Slater's condition implies  solving (finite dimensional) systems of equalities and inequalities, see e.g. \citep{Duffin_RJ_Dantzig_GB_Fan8K_Book_1956,Fan_K_JLinAlg&App_1975_two_Aap_of_consistency,
	Hiebert_K_SIAM_AppAn_1980_sol_eq&ineq,
	Censor_Y_&_ELfving_T_JLinAlg&App_1982_new_meth_Lin_Ineq}. \textcolor{black}{For problems with continuous variables, an approach to check the condition consists in building an {\em initial} pdf  for which equality constraints are satisfied and inequality constraints are satisfied strictly. This approach has been used in Section \ref{sec:example} to verify the fulfillment of Assumption \ref{asm:constrraints} for our numerical example.}
\end{Rmk}

We are now ready to introduce the next result, which gives a solution to Problem \ref{prob:Lemma_1}.

\begin{Lem}\label{lem:Constrained_KL}
{Consider Problem \ref{prob:Lemma_1}.} Then:
\begin{enumerate}
\item[{\bf (R1)}] {the problem has a unique solution and this} is given by the pdf
\begin{equation}
\label{eq:Lem1_opt_fu_sol}
\textcolor{black}{f^\ast(\mathbf{z})} = 	
\textcolor{black}{g(\mathbf{z})} \;
\frac{ 	
	e^{- 
		\left \lbrace
		\alpha \left( \mathbf{z} \right) +
		\sum_{j \in \textcolor{black}{\mathcal{I}_a} \left( \textcolor{black}{f^\ast(\mathbf{z})}  \right) \smallsetminus \lbrace 0 \rbrace }
			\lambda_j^*  h_j \left( \mathbf{z} \right)
		\right \rbrace 
	}
}{e^{1+\lambda^*_0}}.
\end{equation}	
 In \eqref{eq:Lem1_opt_fu_sol}, $\lambda^\ast_j$ is the \textcolor{black}{Lagrange multiplier} associated to the constraint $c_j\left[\textcolor{black}{f(\mathbf{z})} \right]$ and $ \textcolor{black}{\mathcal{I}_a}\left(\textcolor{black}{f^\ast(\mathbf{z})} \right)$ is the active index set defined as
\begin{equation}
\textcolor{black}{\mathcal{I}_a} \left(f\right):
	= \mathcal{E}_0 \cup \lbrace j \in \mathcal{I}: c_j\left[ f \right] =0 \rbrace .
\end{equation}
In (\ref{eq:Lem1_opt_fu_sol}) the \textcolor{black}{Lagrange multipliers} 
$$
\boldsymbol{\lambda}^\ast:= [\lambda_0^\ast,\lambda_1^\ast,\dots,\lambda_{n_e + n_{\textcolor{black}{l}}}^\ast]^T,
$$ 
can be computed by solving the optimization problem
{\begin{equation}
\label{eq:probl_dual_Lemma_1_statement} 
\begin{array}{lcll} 
\boldsymbol{\lambda}^\ast \in \arg
	& \underset{\boldsymbol{\lambda}}{ \max } 	
		& \mathcal{L}^D \left(\boldsymbol{\lambda}\right) \\	
	& \textnormal{s.t.:} 								
		& \lambda_j \textnormal{ free}, 	
			& \forall j \in \mathcal{E}_0,  \\
	& 		
		& \lambda_j \geq 0,  		
			& \forall j \in \mathcal{I} \\ 	
\end{array} 
\end{equation}
where 
\textcolor{black}{\begin{equation}
\label{eq:Lagr_dual_final}
	\begin{array}{llll}
	\mathcal{L}^D \left(\boldsymbol{\lambda}\right):= 
	-  \langle\boldsymbol{\lambda},\mathbf{H}\rangle 
	- \int 
	g   \left(\mathbf{z}\right)  \;
	e^{-\lbrace 1+ \alpha(\mathbf{z}) + \langle \boldsymbol{\lambda},\mathbf{h}\left( \mathbf{z} \right) \rangle \rbrace}
	d\mathbf{z} .
	\end{array} 	
\end{equation}}
}


\item[{\bf (R2)}] Moreover, the corresponding minimum  is $ \mathcal{L} \left( \textcolor{black}{f^\ast(\mathbf{z})}\right)  
= - \left( 
	1 + \sum_{j \in \textcolor{black}{\mathcal{I}_a} \left( \textcolor{black}{f^\ast(\mathbf{z})}  \right) } 
	\lambda^\ast_j \, H_j	
	\right)$.
\end{enumerate}
\end{Lem}
\begin{proof}
See the appendix.
\end{proof}

Before introducing the next technical result, we make the following remarks on Lemma \ref{lem:Constrained_KL}.

\begin{Rmk}\label{rem:constraints}
The equality constraints in \eqref{eq:constr_c_j} can be used to impose parametric prescriptions on the solution. For example, one could impose that $f^{\ast}(\mathbf{z})$ has the central moment of order $i$ equal to some $m_\mathbf{Z}^i$. 
This is equivalent to impose that the solution satisfies $\mathbb{E}_{f}\left[ \mathbf{Z}^i \right] = m_\mathbf{Z}^i$, which in turn can be formalized as {$c_i \left[f(\mathbf{z})\right]:=
	\int{f \left( \mathbf{z} \right) \;	
		\mathbf{z}^i \; d\mathbf{z}} 
	- 	m^i_\mathbf{Z}$.
} \textcolor{black}{These types of equality constraints, also arise within the literature on the approximation of spectral density functions with respect to the KL-divergence under moment constraints, where linear systems are typically considered. See e.g. \citep{1246014,1618839, 8359301}.}
\end{Rmk}

\begin{Rmk}\label{rem:bound_constr}
The inequality constraints in \eqref{eq:constr_c_j} can also be used to assign properties to the  solution: with these constraints, one could express bounds on the expected value of any function of $\mathbf{Z}$, say $h(\mathbf{Z})$. For instance, the rectangular bound $ \underline{m}^2_\mathbf{Z}\leq \mathbb{E}_{f}\left[ \mathbf{Z}^2 \right] \leq \overline{m}^2_\mathbf{Z} $ can be formalized with the pair of inequality constraints:
{		\begin{equation*}
		\label{eq:constr_rect_shape}
		\left\{\begin{array}{ll}
			c_a \left[{f(\mathbf{z})}\right]:=
			\int{f 
				\left( \mathbf{z} \right) \;	
				\mathbf{z}^2 
				\; d\mathbf{z}} 
			-  \;	\overline{m}^2_\mathbf{Z}, \\	
			c_b \left[{f(\mathbf{z})}\right]:=
			-\int{f 
				\left( \mathbf{z} \right) \;	
					\mathbf{z}^2 
				\; d\mathbf{z}} 
			 +  \; \underline{m}^2_\mathbf{Z} .
		\end{array}\right.
	\end{equation*}
}
\end{Rmk}
{\begin{Rmk}
\label{rem:lambda_0}
The  \textcolor{black}{Lagrange multiplier} {$\lambda_0$} in \eqref{eq:Lem1_opt_fu_sol} can be expressed as a function of all the other  \textcolor{black}{Lagrange multipliers}. This can be done by imposing the normalization constraint, i.e. $c_0\left[\textcolor{black}{f^\ast(\mathbf{z})}\right]=0$, on the pdf solving the problem, i.e. $\eqref{eq:Lem1_opt_fu_sol}$. This yields:
	\textcolor{black}{\begin{equation*}
	\begin{split}
	1 +  \lambda_0 =
	 \ln 
	\left( 	
	\int
	g \left(\mathbf{z}\right) 
	\;
	e^{- 
		\left \lbrace
		\alpha \left( \mathbf{z} \right) +
		\sum_{j \in  {\mathcal{I}_a} \left( {f^\ast(\mathbf{z})} \right) \smallsetminus \lbrace 0 \rbrace }
		{\lambda_j} h_j \left( \mathbf{z} \right)
		\right \rbrace 
	}
	d\mathbf{z}
	\right),
	\end{split}
	\end{equation*}	}
 	which can also be obtained by imposing the {stationarity} condition of the Lagrange dual function with respect to $\lambda_0$.
	Also, the  expression for $\lambda_0$ above could be directly embedded in the dual cost function (\ref{eq:Lagr_dual_final}), yielding \textcolor{black}{$\mathcal{L}^{D} \left( \boldsymbol{\lambda} \right)
			 = - \sum_{j \in \mathcal{E} \cup  \mathcal{I} }
			\lambda_j H_j	\,  -  \ln 
			\left( 	
			\int
			g \left(\mathbf{z}\right) 
			\;
			e^{- 
				\left \lbrace
				\alpha \left( \mathbf{z} \right) +
				\sum_{ j \in \mathcal{E} \cup  \mathcal{I} }
				\lambda_j h_j \left( \mathbf{z} \right)
				\right \rbrace 
			}
			d\mathbf{z}
			\right)$},
	and thus reducing by one the dimension of the search space of the dual problem.
\end{Rmk} 
}

We now introduce the following technical result that is also used in the proof of Theorem \ref{theo:ctrlConstr}.
\begin{Lem}\label{proper:KLdiv_split_nD_fn}	
Let $f^n$ and $g^n$ be the \textcolor{black}{joint} pdfs defined in 
\textcolor{black}{\eqref{eq:CL_red_v1_short}} and \eqref{eq:CL_red_ref_short}, respectively. Then:
\textcolor{black}{\begin{equation*}
\begin{split}
\mathcal{D}_{\KL}
\left( f^n || \Ipdf^n \right)  		
& = \mathcal{D}_{\KL}
\left( f^{n-1} ||\Ipdf^{n-1} \right)  \\
& +
\mathbb{E}_{f^{n-1}}
\left[	
\mathcal{D}_{\KL} 
\left(
\tilde{f}^{n}|| \tilde{\Ipdf}^{n}
\right)
\right]	,
\end{split}
\end{equation*}}
\textcolor{black}{where $\tilde{f}^{n}$ and $\tilde{\Ipdf}^{n}$ are the conditional pdfs defined as in Section \ref{ss:Assumpt}, i.e. $\tilde{f}^{n}:= \tilde f_{\mathbf{X}}^n \tilde f_{\mathbf{U}}^n$ and $\tilde{g}^{n}:= \tilde g_{\mathbf{X}}^n \tilde g_{\mathbf{U}}^n$.}	
\end{Lem}
\begin{proof} 
The result is obtained from Property~\ref{proper:KLsplit} (see the appendix for a proof of this property) by setting
$\mathbf{Y} := [\mathbf{X}_0,\mathbf{U}_1,\mathbf{X}_1,\dots, \mathbf{U}_{n-1}, \mathbf{X}_{n-1} ]$ 
and
$\mathbf{Z}:=[\mathbf{U}_{n}, \mathbf{X}_{n} ]$.
\end{proof}

The main result behind the algorithm of Section \ref{sec:algorithm}, the proof of which makes use of the above technical results, is presented next.
\begin{Thm}
\label{theo:ctrlConstr}
{Consider Problem \ref{prob:Main_Constr_Ctrl}.} Then:
	\item[{\bf (R1)}]
	The control policy at time instant $t_k$,  $\left( \tilde{f}_\mathbf{U}^k \right)^\ast$, 
	composing 
	$\left \lbrace 
	\left( \tilde{f}_\mathbf{U}^k \right)^* 
	\right\rbrace_{\mathcal{K}}$ 
	solving the problem is given by
\begin{equation}
\label{eq:opt_ctrl}
\left( \tilde{f}_\mathbf{U}^k \right)^\ast = 
\tilde{\Ipdf}_\mathbf{U}^k 
\frac{e^{-\lbrace 
	\hat{\omega} \left( \mathbf{u}_k,\,\mathbf{x}_{k-1} \right) +
	\sum_{j \in \textcolor{black}{\mathcal{I}_a}^k  \smallsetminus \lbrace 0 \rbrace }
	\left( \lambda_{\mathbf{u},j}^k \right)^\ast
	h_{\mathbf{u},j}^k \left( \mathbf{u}_k \right)	
	\rbrace}}{e^{1+ \left( \lambda_{\mathbf{u},0}^k \right)^* } },
\end{equation}
where:
\begin{itemize}
\item $\hat{\omega}(\cdot,\cdot)$ is generated via the backward recursion
	\begin{equation}
		\label{eq:omega_defdef}
			\hat{\omega} \left( \mathbf{u}_k,\,\mathbf{x}_{k-1} \right)  = 	
			\hat{\alpha} \left( \mathbf{u}_k,\,\mathbf{x}_{k-1} \right) +
			\hat{\beta} \left( \mathbf{u}_k,\,\mathbf{x}_{k-1} \right), 
		\end{equation}
with
\begin{equation}
\label{eq:def:alpha_beta}
\begin{split}
\hat{\alpha} \left( \mathbf{u}_k,\,\mathbf{x}_{k-1} \right)
	& :=
	\mathcal{D}_{\KL}
			\left(
			\tilde{f}^k_{\mathbf{X}}
			|| 
			\tilde{\Ipdf}^k_{\mathbf{X}}
			\right),\\
	\hat{\beta} \left( \mathbf{u}_k,\,\mathbf{x}_{k-1} \right)
	& := 
			- \mathbb{E}_{\tilde{f}^k_\mathbf{X}}
			\left[	
			\ln 
			\hat{\gamma}
			\left(
			\mathbf{X}_{\textcolor{black}{k}}
			\right)
			\right],
		\end{split}
	\end{equation}
and terminal conditions $\hat{\beta} \left( \mathbf{u}_n,\,\mathbf{x}_{n-1} \right) = 0$, $\hat{\alpha} \left( \mathbf{u}_n,\,\mathbf{x}_{n-1} \right) = \mathcal{D}_{\KL}\left(\tilde{f}^n_{\mathbf{X}}||\tilde{\Ipdf}^n_{\mathbf{X}}
\right)$;
\item $\hat{\gamma}\left( \cdot \right)$ in \eqref{eq:def:alpha_beta} is defined as
\begin{equation}
	\label{eq:def_gamma}
	\ln 
	\hat{\gamma}
	\left(
	\mathbf{x}_{k-1}
	\right)
	:=	
	\sum_{j \in \textcolor{black}{\mathcal{I}_a}^k }
	\ln
	\left(
	\hat{\gamma}_{\mathbf{u},j}^k				
	\left(
	\mathbf{x}_{k-1}
	\right)
	\right), 
\end{equation}	
with
\begin{equation}
\label{eq:def_gamma_0}
\hat{\gamma}_{\mathbf{u},0}^k 
\left( \mathbf{x}_{k-1} \right)
= \exp{ \lbrace	
		\left( \lambda_{\mathbf{u},0}^k\right)^* + 1 
		\rbrace 
		}, \ \ \hat{\gamma}_{\boldsymbol{u},0}^{n+1}\left(\mathbf{x}_{n} \right) = 1,
\end{equation}
and
\begin{equation}
\label{eq:def_gamma_i}
\begin{array}{l}
\hat{\gamma}_{\mathbf{u},j}^k				
\left(
\mathbf{x}_{k-1}
\right) :=  
\exp{ \lbrace \left(\lambda_{\mathbf{u},j}^k\right)^\ast H_{\mathbf{u},j}^k	\rbrace}, \ \hat{\gamma}_{\mathbf{u},{j}}^{n+1} \left(\mathbf{x}_{n} \right) = 1,
\end{array}
\end{equation}
$\forall j \in \mathcal{E}^k \cup  \mathcal{I}^k$;	
\item 
{$\left(\lambda_{\mathbf{u},j}^k\right)^\ast$ in (\ref{eq:opt_ctrl}) 
is the Lagrange multiplier associated to the constraint 
$c_{\mathbf{u},j}^k$, while $\textcolor{black}{\mathcal{I}_a}^k$ is the active index set associated to $\left( \tilde{f}_\mathbf{U}^k \right)^*$. 
In particular, the vector of  \textcolor{black}{Lagrange multipliers} 
$\left( \boldsymbol{\lambda}_{\mathbf{u}}^k \right)^*:=
\left[\left(\lambda_{\mathbf{u},0}^k\right)^\ast, 
\left(\lambda_{\mathbf{u},1}^k\right)^\ast, 
\ldots, \left(\lambda_{\mathbf{u},n_e^k+n_l^k}^k\right)^\ast \right]^T$ can be computed solving:
\begin{equation}
\label{eq:probl_dual_k_Thm_1}
\begin{array}{llll} 
\left( \boldsymbol{\lambda}_{\mathbf{u}}^k\right)^* \in \arg 
&   \underset{ \boldsymbol{\lambda}_ \mathbf{u}^k }{ \max } & \mathcal{L}^{D} \left( \boldsymbol{\lambda}_{\mathbf{u}}^k  \right) \\	
& \textnormal{s.t.:} 						& \lambda_{\mathbf{u},j}^k \textnormal{ free}, 	& \forall j \in \mathcal{E}^k  \\
& 											& \lambda_{\mathbf{u},j}^k \geq 0,  		& \forall j \in \mathcal{I}^k, \\ 	
\end{array} 
\end{equation}
where 
\begin{equation*}
\begin{array}{lll}
\mathcal{L}^{D} \left( \boldsymbol{\lambda}_{\mathbf{u}}^k  \right) &=
- \sum_{j \in \mathcal{E}^k \cup  \mathcal{I}^k }
\lambda_{\mathbf{u},j}^k  H_{\mathbf{u},j}^k \,- \\
&
\eqlmarg \ln 
\left( 	
\int
\tilde{\Ipdf}_\mathbf{U}^k
\;
e^{
	-\hat{\omega} \left( \mathbf{u}_k,\,\mathbf{x}_{k-1} \right) 
}
\right.\\
&
\eqlmarg \quad
\left.
e^{- 
	\left \lbrace
	\sum_{ j \in \mathcal{E}^k \cup  \mathcal{I}^k }
	\lambda_{\mathbf{u},j}^k \, h_{\mathbf{u},j}^k \left( \mathbf{u}_k \right)
	\right \rbrace 
}
d\mathbf{u}_k
\right),
\end{array}
\end{equation*}
}
with $\left(\lambda_{\mathbf{u},0}^k\right)^\ast$ given by
\textcolor{black}{\begin{equation}
\label{eqn:lambda0_thm}
\begin{split}
\left(\lambda_{\mathbf{u},0}^k\right)^\ast 	& = 
\ln
\left(
\int \tilde{\Ipdf}_\mathbf{U}^k 
e^{-\hat{\omega} \left( \mathbf{u}_k,\,\mathbf{x}_{k-1} \right) }
\right.\\
& \left.
e^{-\left\lbrace 	
	\sum_{j \in {\mathcal{I}_a}^k \smallsetminus \lbrace 0 \rbrace }
	\left( \lambda_{\mathbf{u},j}^k \right)^\ast
	\, h_{\mathbf{u},j}^k \left( \mathbf{u}_k \right)
	\right\rbrace}
	\, 
	d\mathbf{u}_k 
\right) -1.
\end{split}
\end{equation}}
\end{itemize}	
\item[{\bf (R2)}] Moreover, the corresponding minimum at time  $t_k$ is given by:
\begin{equation}
\label{eq:ctrl_contrst_minVal}
B^*_k 
:= 	
- \mathbb{E}_{p^{k-1}_\mathbf{X}}
\left[
\ln \hat{\gamma} \left( \mathbf{X}_{k-1} \right) 	
\right]. 	
\end{equation}
where $p_{\mathbf{X}}^{k}$ denotes the pdf of the state at time $t_k$ 
(i.e. $p_{\mathbf{X}}^{k}:= f\left(\mathbf{x}_{k}\right)$). 
\end{Thm}
Before giving the proof we make the following \textcolor{black}{remark}:
{\begin{Rmk}
The policy solving Problem \ref{prob:Main_Constr_Ctrl}, i.e. $\left(\tilde f_{\mathbf{U}}^k \right)^\ast$, directly depends on $\tilde g_{\mathbf{U}}^k$. This is the policy extracted from the examples and a natural choice is to estimate it from the \textcolor{black}{available data} (see also the example in Section \ref{sec:example}).  In principle,  one could use a $\tilde g_{\mathbf{U}}^k$ that, while not being extracted from the examples, embeds design preferences that might be known a-priori. \textcolor{black}{If full knowledge of the system (and of its constraints) is available, one can build a {\em synthetic} joint pdf $g^n$ that embeds the desired design properties. While our results can be used in this {\em ideal} situation, we remark here that building such synthetic pdf is not always possible in situations of practical interest where the example \textcolor{black}{data} are collected from a system that is not the same and/or does not have the same constraints as the system under control (or  the constraints of the system under control are not fully known).}  Finally, another choice, for pdfs having compact supports, is to set $\tilde g_{\mathbf{U}}^k$ equal to the uniform distribution.  
\end{Rmk}}
\begin{proof}
we prove the result by induction and the proof is organized in steps. First, in {\bf Step 1}, we leverage Lemma \ref{proper:KLdiv_split_nD_fn} to show that Problem \ref{prob:Main_Constr_Ctrl} can be split into sub-problems, where the optimization sub-problem for the last iteration, i.e. $k = {n}$, can be solved independently on the others. We then ({\bf Step 2}) make use of Lemma \ref{lem:Constrained_KL} to find an explicit solution for the sub-problem at $k={n}$. Once this is done, we update the cost  of Problem \ref{prob:Main_Constr_Ctrl} with the minimum found by solving the sub-problem at $k=n$ and show, in {\bf Step 3}, that the original problem can be again broken down into sub-problems. This time, the sub-problem at iteration $k={n-1}$ can be solved independently on the others. We then solve this sub-problem and note, in {\bf Step 4}, how \textcolor{black}{for all the remaining time-steps} the structure of the optimization remains the same. From this, the desired conclusions are drawn.

Before proceeding with the proof note that, for notational convenience, we use the shorthand notation $\left \lbrace
\constrVec{\mathbf{u}}{k}
\right \rbrace$ to denote the set of constraints of Problem \ref{prob:Main_Constr_Ctrl} at iteration $k$. We also denote by $\left \lbrace
\constrVec{\mathbf{u}}{k}\right\rbrace_{\mathcal{K}}$ the set of constraints over \textcolor{black}{$\mathcal{K}$ and by $\left \lbrace
\constrVec{\mathbf{u}}{k}
\right \rbrace_{k = 1}^{n-1}$ the constraints from  $t_1$ up to $t_{n-1}$. }

{\bf Step 1.} Note that, following Lemma \ref{proper:KLdiv_split_nD_fn}, Problem \ref{prob:Main_Constr_Ctrl} can be re-written as
{
\begin{equation}
\label{eq:SplitLast_red}
\begin{array}{ll}
	\begin{array}{cl}
	\underset{\left\lbrace
				\tilde{f}_\mathbf{U}^k
				\right\rbrace_{\mathcal{K}}
				}{\min} 
			& \mathcal{D}_{\KL} \left( f^{n} || \Ipdf^{n} \right)  = \\
	\text{ s.t.:} 
			&  \left \lbrace \constrVec{\mathbf{u}}{k} \right \rbrace_{\mathcal{K}}
	\end{array} \\
	\eqlmarg
	\begin{array}{lcl}	
	 = &	\underset{\left\lbrace
				\tilde{f}_\mathbf{U}^k
				\right\rbrace_{k = 1}^{n-1}}{\min} 
			& \left\lbrace 	\mathcal{D}_{\KL} \left( f^{n-1} || \Ipdf^{n-1} \right) + B^\ast_n 	\right\rbrace \\
	   & \text{ s.t.:} 
	 		&  \left \lbrace \constrVec{\mathbf{u}}{k} \right \rbrace_{k = 1}^{n-1}
	\end{array} 
\end{array}
\end{equation}
}
where:
{
\begin{subequations}
\textcolor{black}{\begin{equation}
\label{eq:BnStar_Ctrl_constr}
	B^\ast_n := 
		\underset{\tilde{f}_{\mathbf{U}}^n}{\min} 
		 B_n 
		 \ \ \ 
				\text{ s.t.:} \   \constrVec{\mathbf{u}}{n}
\end{equation}}
and
\begin{equation}
B_n 
:=
\mathbb{E}_{f^{n-1}}
\left[
\mathcal{D}_{\KL}
\left(
\tilde{f}^n ||
\tilde{\Ipdf}^n
\right)
\right].
\end{equation}
\end{subequations}
%
}

That is, Problem \ref{prob:Main_Constr_Ctrl} can be approached by solving first the optimization of the last iteration of the horizon $\mathcal{K}$ 
(the term $B_n$ in \eqref{eq:SplitLast_red}) and then by taking into account the result from this optimization problem in the optimization up to iteration ${n-1}$. 

{\bf Step 2.} We first observe that, for $B_n$ defined in (\ref{eq:BnStar_Ctrl_constr}):
\begin{equation*}
B_n 
=
\mathbb{E}_{f^{n-1}}
\left[
\mathcal{D}_{\KL}
\left(
\tilde{f}^n ||
\tilde{\Ipdf}^n
\right)
\right]
=
\mathbb{E}_{p^{n-1}_\mathbf{X}}
\left[
\mathcal{D}_{\KL}
\left(
\tilde{f}^n ||
\tilde{\Ipdf}^n
\right)
\right].
\end{equation*}
The above expression was obtained by \textcolor{black}{noticing that, by definition of $\tilde{f}^n$ and $\tilde{g}^n$ (see soon before (\ref{eq:CL_red_v1_short}) and (\ref{eq:CL_red_ref_short})),} $\mathcal{D}_{\KL} \left( \tilde{f}^n || \tilde{\Ipdf}^n \right)$ is  a function of the previous state and, to stress this in the notation, we let $\hat{A}\left( \cdot \right):=\mathcal{D}_{\KL} \left( \tilde{f}^n || \tilde{\Ipdf}^n \right)$. Hence, $B_n$ can be written as 
\begin{equation}
\label{eq:Bn_Ctrl_constr}
B_n 
=
\mathbb{E}_{p^{n-1}_\mathbf{X}}
\left[
\mathcal{D}_{\KL}
\left(
\tilde{f}^n ||
\tilde{\Ipdf}^n
\right)
\right]
=
\mathbb{E}_{p^{n-1}_\mathbf{X}}
\left[ \hat{A}\left(\mathbf{X}_{n-1}\right) \right],
\end{equation}
and the sub-problem in (\ref{eq:BnStar_Ctrl_constr}) becomes:
\textcolor{black}{\begin{equation}\label{eq:Prob_Bn_StarCtrl_constr}
B^\ast_n = \underset{\tilde{f}_{\mathbf{U}}^n}{\min} 
		\mathbb{E}_{p^{n-1}_\mathbf{X}} \left[ \hat{A}\left(\mathbf{X}_{n-1}\right) \right] 
	\ \ \  \text{s.t.:}   \ \constrVec{\mathbf{u}}{n} 
\end{equation}
}
Now, note that
{\begin{equation}
\label{eqn:equalityproof_thm}
\begin{array}{cll}
	\underset{\tilde{f}_{\mathbf{U}}^n}{\min} 
		& \mathbb{E}_{p^{n-1}_\mathbf{X}}\left[ \hat{A}\left(\mathbf{X}_{n-1}\right) \right]  
			& = \mathbb{E}_{p^{n-1}_\mathbf{X}}\left[ A^\ast_n \right],\\
 \text{ s.t.:} &  \constrVec{\mathbf{u}}{n}
\end{array} 
\end{equation} 
}
where 
\textcolor{black}{\begin{equation}
\label{eq:AnStars_Ctrl_constr}
A^{\ast}_n  := \underset{\tilde{f}_\mathbf{u}^n}{\min} \hat{A}(\mathbf{x}_{n-1}) \\
	 \ \ \ \text{s.t.:} \  \constrVec{\mathbf{u}}{n} 
\end{equation}
}
Also, the equality in (\ref{eqn:equalityproof_thm}) was obtained by using the fact that the expectation operator is linear and the fact that the decision variable 
(i.e. $\tilde{f}_\mathbf{U}^n$) is independent on the pdf over which the expectation is performed (i.e. $p^{n-1}_\mathbf{X}$). 

Following (\ref{eqn:equalityproof_thm}), we can obtain $B_n^\ast$ by solving \eqref{eq:AnStars_Ctrl_constr} and then by averaging $A_n^\ast$ over $p_{\mathbf{X}}^{n-1}$. We now focus on solving problem \eqref{eq:AnStars_Ctrl_constr}. From \eqref{eq:Bn_Ctrl_constr}, we get:
\begin{subequations}
	\begin{equation}
	\label{eq:An_expl_of_alpha}
	\hat{A}\left(\mathbf{x}_{n-1}\right)
	=
	\int
	\tilde{f}^n_\mathbf{U}
	\left[
	\ln
	\left(
	\frac{
		\tilde{f}^n_\mathbf{U}
	}{
		\tilde{\Ipdf}^n_\mathbf{U}
	}
	\right)
	+
	\hat{\alpha}
	\left( \mathbf{u}_n,\mathbf{x}_{n-1} \right)
	\right]\,
	d\mathbf{u}_{n},
	\end{equation}	
	\begin{equation}	\label{eq:alpha_time_n}
	\hat{\alpha}
	\left( \mathbf{u}_n,\mathbf{x}_{n-1} \right)
	:=
	\mathcal{D}_{\KL} 
	\left( \tilde{f}^n_\mathbf{X} || \tilde{\Ipdf}^n_\mathbf{X} \right).
	\end{equation}
\end{subequations}
In turn, \eqref{eq:An_expl_of_alpha} can be compactly written as:
\begin{equation*}
\hat{A}(\mathbf{x}_{n-1}) 
= 
\mathcal{D}_{\KL} \left( \tilde{f}_\mathbf{U}^n || \tilde{\Ipdf}_\mathbf{U}^n \right) 
+\int{\tilde{f}_\mathbf{U}^n \,
	\hat{\alpha} \left( \mathbf{u}_n,\mathbf{x}_{n-1} \right) 	\,
	d\mathbf{u}_n},
\end{equation*}
where we used the definition of \KL-divergence.

Hence, Lemma \ref{lem:Constrained_KL} can be used to solve the optimization problem in \eqref{eq:AnStars_Ctrl_constr}. 
Indeed by applying Lemma~\ref{lem:Constrained_KL} with: 
$\mathbf{Z}=\mathbf{U}_n$, 
$\textcolor{black}{f(\mathbf{z})}=\tilde{f}_\mathbf{U}^n$, $\textcolor{black}{g(\mathbf{z})}=\tilde{g}_\mathbf{U}^n$,
$\alpha(\cdot) = \hat{\alpha}(\cdot,\mathbf{x}_{n-1} )$,
$h_j(\mathbf{z}) = h_{\mathbf{u},j}^n(\mathbf{u}_n)$,
$H_j = H_{\mathbf{u},j}^n$,
$c_j [\cdot]= c_{\mathbf{u},j}^n [\cdot] $,
$\lambda_{j} = \lambda_{\mathbf{u},j}^n$, 
$\mathcal{E} = \mathcal{E}^n$,
$\mathcal{I} = \mathcal{I}^n$,
we get the following solution to \eqref{eq:AnStars_Ctrl_constr}:
\begin{equation}
\label{eq:opt_ctrl_n}
\left( \tilde{f}_\mathbf{U}^n \right)^\ast =
\tilde{\Ipdf}_\mathbf{U}^n 
\frac{e^{-\lbrace 
		\hat{\alpha} \left( \mathbf{u}_n,\,\mathbf{x}_{n-1} \right) +
		\sum_{j \in \textcolor{black}{\mathcal{I}_a}^n  \smallsetminus \lbrace 0 \rbrace }
		\left( \lambda_{\mathbf{u},j}^n \right)^\ast
		h_{\mathbf{u},j}^n \left( \mathbf{u}_n \right)	
		\rbrace}}{e^{1+ \left( \lambda_{\mathbf{u},0}^n \right)^* } },
\end{equation}
where $\textcolor{black}{\mathcal{I}_a}^n$ is the active set index associated to {$\left(\tilde f_{\mathbf{U}}^n\right)^\ast$}.  
In the above pdf, $\left( \lambda_{\mathbf{u},j}^n \right)^\ast, j \in \mathcal{E}_0^n \cup \mathcal{I}^n$ are the  \textcolor{black}{Lagrange multipliers} at the last  iteration $k =n$. Now, following Lemma \ref{lem:Constrained_KL} and Remark \ref{rem:lambda_0},  {the  \textcolor{black}{Lagrange multipliers} 
$\left( \boldsymbol{\lambda}_{\mathbf{u}}^n \right)^\ast = 
\left[\left(\lambda_{\mathbf{u},0}^n\right)^\ast,
\left(\lambda_{\mathbf{u},1}^n\right)^\ast, \ldots, \left(\lambda_{\mathbf{u},n_e^n+n_l^n}^n\right)^\ast \right]^T$ are computed by solving
}
\begin{equation*}
\begin{array}{llll} 
{ \left( \boldsymbol{\lambda}_{\mathbf{u}}^n \right)^\ast \in \arg }
& \underset{ }{ \max } 	& \mathcal{L}^{D} \left(  \boldsymbol{\lambda}_{\mathbf{u}}^n\right) \\	
& \text{s.t.:} 								& \lambda_{\mathbf{u},j}^n \text{ free}, 	& { \forall j \in \mathcal{E}^n }, \\
& 											& \lambda_{\mathbf{u},j}^n \geq 0,  		& { \forall j \in \mathcal{I}^n }\\ 	
\end{array} 
\end{equation*}
{choosing $\left[\left(\lambda_{\mathbf{u},1}^n\right)^\ast, \ldots, \left(\lambda_{\mathbf{u},n_e^n+n_l^n}^n\right)^\ast \right]^T$ so that $\left(\tilde f_{\mathbf{U}}^n\right)^\ast$ is feasible.} In the above expression \textcolor{black}{$\mathcal{L}^{D} 
\left( \boldsymbol{\lambda}_{\mathbf{u}}^n\right)$ is defined as:}
\textcolor{black}{\begin{equation*}
\begin{split}
\mathcal{L}^{D} 
\left( \boldsymbol{\lambda}_{\mathbf{u}}^n\right)  
& = 
- \sum_{j \in \mathcal{E}^n \cup  \mathcal{I}^n } 
\lambda_{\mathbf{u},j}^n  
H_{\mathbf{u},j}^n 	\,-
 \ln 
\left( 	
\int
\tilde{\Ipdf}_\mathbf{U}^n
\;
e^{- 
	\hat{\alpha} \left( \mathbf{u}_n,\,\mathbf{x}_{n-1} \right) 
} 
\right. \\
& \cdot
\left.
e^{- 
	\left \lbrace
	\sum_{ j \in \mathcal{E}^{{n}} \cup  \mathcal{I}^{{n}} }
	\lambda_{\mathbf{u},j}^n  \,
	h_{\mathbf{u},j}^n \left( \mathbf{z} \right)
	\right \rbrace 
}
d\mathbf{u}_n \right),
\end{split}
\end{equation*}
}
{and $\left( \lambda_{\mathbf{u},0}^n \right)^\ast$ can be obtained from all the other  \textcolor{black}{Lagrange multipliers} by normalizing (\ref{eq:opt_ctrl_n}), i.e.: }
{\begin{equation*}
	\begin{array}{ll}
		\left( \lambda_{\mathbf{u},0}^n \right)^\ast + 1 
	& = \ln 
		\left( 
		\int 
		\tilde{\Ipdf}_\mathbf{U}^n 
		e^{- \hat{\alpha} \left( \mathbf{u}_n,\,\mathbf{x}_{n-1} \right) }	
		\right. \\
	&	\eqlmarg \;
		\left. 
		e^{ -\left \lbrace 
			\sum_{j \in \textcolor{black}{\mathcal{I}_a}^n  \smallsetminus \lbrace 0 \rbrace }
			\left( \lambda_{\mathbf{u},j}^n \right)^\ast
			h_{\mathbf{u},j}^n \left( \mathbf{u}_n \right)	
			\right \rbrace }	
		\; d\mathbf{u}_n  
		\right)
		\\
	& = \ln 
		\left(
		\hat{\gamma}_{\mathbf{u},0}^n	
		\left( \mathbf{x}_{n-1} \right)
		\right) .
	\end{array}
\end{equation*}}
Moreover, from Lemma \ref{lem:Constrained_KL}, the minimum of the problem in \eqref{eq:AnStars_Ctrl_constr} is given by:
\begin{equation*}
\hat{A}^\ast_n =
- 	\left(
1 + \sum_{j \in \textcolor{black}{\mathcal{I}_a}^n} 
	\left(
		\lambda_{\mathbf{u},j}^n\right)^\ast H_{\mathbf{u},j}^n	
	\right),
\end{equation*}
which, using the definitions in \eqref{eq:def_gamma_0} and \eqref{eq:def_gamma_i}, can be equivalently written as
\begin{equation*}
\hat{A}^\ast_n
=
-
\left[
\sum_{j \in \textcolor{black}{\mathcal{I}_a}^n} 
\ln
\left(
\hat{\gamma}_{\mathbf{u},j}^n				
\left(
\mathbf{x}_{n-1}
\right)
\right) 	
\right] 
=
- \ln \hat{\gamma} 
	\left(
	\mathbf{x}_{n-1}
	\right).
\end{equation*}
Thus, we get:
\begin{equation}
\label{eq:ctrl_contrst_min_Bn}
B^\ast_n
=
- \mathbb{E}_{p^{n-1}_\mathbf{X}}
\left[
\ln \hat{\gamma} 
\left(
\mathbf{X}_{n-1}
\right)
\right].
\end{equation}
	
{\bf Step 3.} Note now that the $B^\ast_n$ in (\ref{eq:ctrl_contrst_min_Bn}) only depends on $\mathbf{X}_{n-1}$ and therefore the original problem \eqref{eq:SplitLast_red} can be split, following Lemma \ref{proper:KLdiv_split_nD_fn}, as
{\begin{equation}
\label{eq:Split2Last_red_explicit}
\begin{array}{ll}
\begin{array}{cl}
\underset{\left\lbrace
	\tilde{f}_\mathbf{U}^k
	\right\rbrace_{k = 1}^{n-1}}{\min} 
&\left\lbrace \mathcal{D}_{\KL} \left( f^{n-1} || \Ipdf^{n-1} \right) + B^*_{n} \right\rbrace   \\
\text{ s.t.:} 
&  	\left \lbrace 	\constrVec{\mathbf{u}}{k} \right \rbrace_{k = 1}^{n-1}  
\end{array} \\
\eqlmarg
\begin{array}{lcl}
= & \underset{\left\lbrace 	\tilde{f}_\mathbf{U}^k \right\rbrace_{k = 1}^{n-2}}{\min} 
		&\left\lbrace \mathcal{D}_{\KL} \left( f^{n-2} || \Ipdf^{n-2} \right) + B^*_{n-1} \right\rbrace \\
  &	\text{ s.t.:} 
 		&  	\left \lbrace 	\constrVec{\mathbf{u}}{k} \right \rbrace_{k = 1}^{n-2} 
\end{array}
\end{array}
\end{equation}
}
where:
\begin{subequations}
\textcolor{black}{\begin{equation}
B^\ast_{n-1}  :=\underset{{\tilde{f}_\mathbf{U}^{n-1}}}{\min}  B_{n-1} \ \ \ \textnormal{s.t.:} \ \constrVec{\mathbf{u}}{n-1}
\end{equation} }
and
\begin{equation}
\label{eq:Bnm1_Ctrl_constr}
B_{n-1} 
:=
\mathbb{E}_{f^{n-2}}
\left[
\mathcal{D}_{\KL}
\left(
\tilde{f}^{n-1} ||
\tilde{\Ipdf}^{n-1}
\right)
\right]+ B^\ast_n.
\end{equation}
\end{subequations}
We approach the above problem in the same way we used to solve the problem in (\ref{eq:Prob_Bn_StarCtrl_constr}). We do this by finding a function, $\hat{A}\left(\mathbf{x}_{n-2}\right)$, such that $B_{n-1} 
=
\mathbb{E}_{p^{n-2}_\mathbf{X}}
\left[ \hat{A}\left(\mathbf{X}_{n-2}\right) \right]$. 
Once this is done, we then get
\textcolor{black}{\begin{equation}
\label{eq:Anm1Stars_Ctrl_constr}
A^{\ast}_{n-1} := \underset{\tilde{f}_\mathbf{U}^{n-1}}{\min}   \hat{A}(\mathbf{x}_{n-2}) \ \ \ \text{ s.t.:} \   \constrVec{\mathbf{u}}{n-1}  
\end{equation}
}
and obtain  $B^*_{n-1}$ as $B^\ast_{n-1}
\begin{array}{ll}
:=
\mathbb{E}_{p^{n-2}_\mathbf{X}}
\left[
A^*_{n-1}
\right]
\end{array}$. 
To this end we first note that the following identities 
\begin{subequations}
\begin{equation}
\label{eq:Epn_Epnm_Etilfnm}
\mathbb{E}_{p^{n-1}_\mathbf{X}}
\left[ 
\varphi \left(\mathbf{X}_{n-1} \right)
\right]
=
\mathbb{E}_{p^{n-2}_\mathbf{X}}
\left[ 
\mathbb{E}_{\tilde{f}^{n-1}}
\left[ 
\varphi \left(\mathbf{X}_{n-1} \right)
\right]
\right],
\end{equation}
\begin{equation}
\mathbb{E}_{\tilde{f}^{n-1}_\mathbf{X}}
\left[ 
\varphi \left(\mathbf{X}_{n-1} \right)
\right]
=
\mathbb{E}_{\tilde{f}^{n-2}_\mathbf{X}}
\left[ 
\mathbb{E}_{\tilde{f}^{n-1}}
\left[ 
\varphi \left(\mathbf{X}_{n-1} \right)
\right]
\right],
\end{equation}
\end{subequations}
hold for any function $\varphi$ of $\mathbf{X}_{n-1}$. Therefore, by means of
\eqref{eq:ctrl_contrst_min_Bn} and \eqref{eq:Epn_Epnm_Etilfnm}
we obtain, from \eqref{eq:Bnm1_Ctrl_constr}:
\textcolor{black}{\begin{equation*}
\begin{split}
 & B_{n-1}   =  
 \mathbb{E}_{f^{n-2}}
\left[
\mathcal{D}_{\KL}
\left(
\tilde{f}^{n-1} ||
\tilde{\Ipdf}^{n-1}
\right)
\right]
+
B^\ast_{n} \\
& =
\mathbb{E}_{p^{n-2}_\mathbf{X}}
\left[
\mathcal{D}_{\KL}
\left(
\tilde{f}^{n-1} ||
\tilde{\Ipdf}^{n-1}
\right)
\right]
+
B^\ast_{n}  \\
& =
\mathbb{E}_{p^{n-2}_\mathbf{X}}
\left[
\mathcal{D}_{\KL}
\left( \tilde{f}^{n-1} || \tilde{\Ipdf}^{n-1} \right)
\right]  - 
\mathbb{E}_{p^{n-2}_\mathbf{X}}
\left[
\mathbb{E}_{\tilde{f}^{n-1}}
\left[
\ln \hat{\gamma}\left( \mathbf{X}_{n-1} \right)	
\right]
\right]  \\
& =
 \mathbb{E}_{p^{n-2}_\mathbf{x}}
\left[
\underbrace{
\mathcal{D}_{\KL}
\left(
\tilde{f}^{n-1} ||
\tilde{\Ipdf}^{n-1}
\right)
+
\mathbb{E}_{\tilde{f}^{n-1}}
\left[
- \ln \hat{\gamma}\left( \mathbf{X}_{n-1} \right)	
\right]
}_{=:\hat{A}\left(\mathbf{X}_{n-2}\right)  }
\right],  
\end{split}	
\end{equation*}
} 
and the term $\hat{A}\left(\mathbf{x}_{n-2}\right)$ can hence be recognized.
Now, following the same reasoning we used to obtain $\hat{A}(\mathbf{x}_{n-1})$, we explicitly write $\hat{A}(\mathbf{x}_{n-2})$ in compact form as
{\begin{equation}
\label{eq:Anm1_final}
\begin{array}{l}
 \hat{A}\left(\mathbf{x}_{n-2}\right) =\\ 
	\mathcal{D}_{\KL}
	\left(
	\tilde{f}^{n-1} ||
	\tilde{\Ipdf}^{n-1}
	\right)
	+
	\mathbb{E}_{\tilde{f}^{n-1}_\mathbf{U}}
	\left[
	\mathbb{E}_{\tilde{f}^{n-1}_\mathbf{X}}
	\left[
	-
	\ln \hat{\gamma}\left( \mathbf{X}_{n-1} \right)	
	\right]
	\right] =   \\  
	\int
	\tilde{f}^{n-1}_\mathbf{U}
	\left\lbrace
	\ln
	\left(
	\frac{
		\tilde{f}^{n-1}_\mathbf{U}
	}{
		\tilde{\Ipdf}^{n-1}_\mathbf{U}
	}
	\right)
	+
	\hat{\omega} \left( \mathbf{u}_{n-1},\,\mathbf{x}_{n-2} \right)
	\right\rbrace
	\, d\mathbf{u}_{n-1}\, ,
\end{array}
\end{equation} 
}
where $\hat{\omega} \left( \mathbf{u}_{n-1},\,\mathbf{x}_{n-2} \right) = \hat{\alpha} \left( \mathbf{u}_{n-1},\,\mathbf{x}_{n-2} \right)+  \hat{\beta} \left( \mathbf{u}_{n-1},\,\mathbf{x}_{n-2} \right)$ and
\begin{equation*}
\begin{array}{rll}
\hat{\alpha} \left( \mathbf{u}_{n-1},\,\mathbf{x}_{n-2} \right)
& :=
\mathcal{D}_{\KL} \left( \tilde{f}^{n-1}_\mathbf{X} || \tilde{\Ipdf}^{n-1}_\mathbf{X} \right),	\\		
\hat{\beta} \left( \mathbf{u}_{n-1},\,\mathbf{x}_{n-2} \right)
& :=
- \mathbb{E}_{\tilde{f}^{n-1}_\mathbf{X}}
\left[
\ln\hat{\gamma}\left( \mathbf{X}_{n-1} \right)	
\right].
\end{array}
\end{equation*}
The last expression we found for $\hat{A}(\mathbf{x}_{n-2})$ in \eqref{eq:Anm1_final} enables us to use Lemma~\ref{lem:Constrained_KL} in order to solve the optimization problem in \eqref{eq:Anm1Stars_Ctrl_constr}. This time, by applying Lemma~\ref{lem:Constrained_KL} with $\mathbf{Z}=\mathbf{U}_{n-1}$, 
$\textcolor{black}{f(\mathbf{z})}=\tilde{f}_\mathbf{U}^{n-1}$, $\textcolor{black}{g(\mathbf{z})}=\tilde{g}_\mathbf{U}^{n-1}$,
$\alpha(\cdot) = \hat{\omega}(\cdot,\mathbf{x}_{n-2} )$,
$h_j(\mathbf{z}) = h_{\mathbf{u},j}^{n-1}(\mathbf{u}_{n-1})$,
$H_j = H_{\mathbf{u},j}^{n-1}$,
$c_j[\cdot] = c_{\mathbf{u},j}^{n-1}[\cdot]$,
$\lambda_{j} = \lambda_{\mathbf{u},j}^{n-1}$, 
$\mathcal{E} = \mathcal{E}^{n-1}$,
$\mathcal{I} = \mathcal{I}^{n-1}$,
we get the following solution to \eqref{eq:Anm1Stars_Ctrl_constr}:
\begin{equation*}
\begin{split}
&\left( \tilde{f}_\mathbf{U}^{n-1} \right)^\ast = \\
& \tilde{\Ipdf}_\mathbf{U}^{n-1} 
\frac{e^{-\lbrace 
		\hat{\omega} \left( \mathbf{u}_{n-1},\,\mathbf{x}_{n-2} \right) +
		\sum_{j \in \textcolor{black}{\mathcal{I}_a}^{n-1} \smallsetminus \lbrace 0 \rbrace }
		\left( \lambda_{\mathbf{u},j}^{n-1} \right)^\ast
		h_{\mathbf{u},j}^{n-1} \left( \mathbf{u}_{n-1} \right)	
		\rbrace
}}{ 	e^{1+ \left( \lambda_{\mathbf{u},0}^{n-1} \right)^\ast }	},
\end{split}
\end{equation*}
where $\textcolor{black}{\mathcal{I}_a}^{n-1}$ is the active set index associated to $\left(\tilde f_{\boldsymbol{U}}^{n-1}\right)^{{\ast}}$ and the  \textcolor{black}{Lagrange multipliers} can again be obtained from Lemma \ref{lem:Constrained_KL}. Namely:
\begin{itemize}
\item {
$\left( \boldsymbol{\lambda}_{\mathbf{u}}^{n-1} \right)^* = 
\left[\left(\lambda_{\mathbf{u},0}^{n-1}\right)^\ast,
\left(\lambda_{\mathbf{u},1}^{n-1}\right)^\ast, \ldots, 
\left(\lambda_{\mathbf{u},n_e^n+n_l^n}^{n-1}\right)^\ast \right]^T$} are the solution to
\begin{equation*}
\begin{array}{llll} 
{	\left( \boldsymbol{\lambda}_{\mathbf{u}}^{n-1} \right)^* \in \arg }
& \underset{ }{ \max } 	& \mathcal{L}^{D} \left(  \boldsymbol{\lambda}_{\mathbf{u}}^{n-1}\right) \\	
& \text{s.t.:} 								& \lambda_{\mathbf{u},j}^{n-1} \text{ free}, 	& { \forall j \in \mathcal{E}^{n-1}},  \\
& 											& \lambda_{\mathbf{u},j}^{n-1} \geq 0,  		& { \forall j \in \mathcal{I}^{n-1}} \\ 	
\end{array} 
\end{equation*}
where 
{
\begin{equation*}
\begin{array}{ll}
 & \mathcal{L}^{D} \left( \boldsymbol{\lambda}_{\mathbf{u}}^{n-1} \right) 
   = 
- \sum_{j \in \mathcal{E}^{n-1} \cup  \mathcal{I}^{n-1} }
\lambda_{\mathbf{u},j}^{n-1}  
H_{\mathbf{u},j}^{n-1} 	 \\
& \eqlmarg
 - \ln 
\left( 	
\int
\tilde{\Ipdf}_\mathbf{U}^{n-1} 
\;
e^{- 	
	\hat{\omega} \left( \mathbf{u}_{n-1},\,\mathbf{x}_{n-2} \right)  +
}\right.\\
& \eqlmarg 
\left.
e^{- 
	\left \lbrace
	\sum_{ j \in \mathcal{E}^{\textcolor{black}{n-1}} \cup  \mathcal{I}^{\textcolor{black}{n-1}} }
	\lambda_{\mathbf{u},j}^{n-1}  
	h_{\mathbf{u},j}^{n-1} \left( \mathbf{u}_{n-1} \right)
	\right \rbrace 
}
d\mathbf{u}_{n-1}
\right);
\end{array}
\end{equation*}
}
\item  $ \left( \lambda_{\mathbf{u},0}^{n-1} \right)^\ast $ is given by 
{\begin{equation*}
\begin{array}{ll}
&  \left( \lambda_{\mathbf{u},0}^{n-1} \right)^\ast + 1  
= 
\ln 
\left\lbrace 
\int 
\tilde{\Ipdf}_\mathbf{U}^{n-1} 
e^{-\left\lbrace 
	\hat{\omega} \left( \mathbf{u}_{n-1},\,\mathbf{x}_{n-2} \right) 
	\right\rbrace}
\right. \\	
& \left.	
e^{-\left\lbrace 
	\sum_{j \in \textcolor{black}{\mathcal{I}_a}^{n-1}  \smallsetminus \lbrace 0 
		\rbrace }
	\left( \lambda_{\mathbf{u},j}^{n-1} \right)^\ast
	h_{\mathbf{u},j}^{n-1} \left( \mathbf{u}_{n-1} \right)	
	\right\rbrace}	
\; d\mathbf{u}_{n-1}  
\right\rbrace 
\\
& =\ln 
\left\lbrace
\hat{\gamma}_{\mathbf{u},0}^{n-1}	
\left( \mathbf{x}_{n-2} \right)
\right\rbrace .
\end{array}
\end{equation*}
}
\end{itemize}
Moreover, from Lemma \ref{lem:Constrained_KL} we also obtain:
\begin{equation*}
\begin{array}{rl}
B^\ast_{n-1} 
& = 	 
- \mathbb{E}_{\tilde{f}^{n-2}_\mathbf{X}}
\left[
\sum_{j \in \textcolor{black}{\mathcal{I}_a}^{n-1}}
\ln 
\hat{\gamma}_{\mathbf{u},j}^{n-1} \left( \mathbf{X}_{n-2} \right) 	
\right]  \\
& = 	 
- \mathbb{E}_{\tilde{f}^{n-2}_\mathbf{X}}
\left[
\ln \hat{\gamma} \left( \mathbf{X}_{n-2} \right)	
\right] .
\end{array}  	
\end{equation*}	

{\bf Step 4.} The proof can then be concluded by observing that,  using $B^\ast_{n-1}$ in (\ref{eq:Split2Last_red_explicit}), the optimization can again be split in sub-problems, with the {\em last} sub-problem (i.e. the sub-problem corresponding to $k= {n-2}$) being independent from the others and having the same structure as the problem we solved at $k=n-1$. Hence, the solution at the generic iteration $k$, i.e. $\left(\tilde{f}_\mathbf{U}^{k} \right)^\ast$, will have the same structure as $\left(\tilde{f}_\mathbf{U}^{n-1} \right)^\ast$, with the functions $\hat{\alpha}(\cdot,\cdot)$, $\hat{\beta}(\cdot,\cdot)$, $\hat{\omega}(\cdot)$ given by (\ref{eq:omega_defdef}) - (\ref{eq:def:alpha_beta})  and the  \textcolor{black}{Lagrange multipliers} given by (\ref{eq:probl_dual_k_Thm_1}) - (\ref{eqn:lambda0_thm}). This leads to the expression in (\ref{eq:opt_ctrl}) with the minimum given in (\ref{eq:ctrl_contrst_minVal}). Finally, for the last iteration ($k=n$) we note from (\ref{eq:alpha_time_n}) that $\hat{\beta} \left( \mathbf{u}_{n},\,\mathbf{x}_{n-1} \right) = 0$ which, by means of (\ref{eq:def_gamma}), implies that $\hat{\gamma}_{\mathbf{u},{j}}^{n+1} \left(\mathbf{x}_{n} \right) = 1$, $\forall j$. This gives the terminal conditions in (\ref{eq:def_gamma_0}) and (\ref{eq:def_gamma_i}). The proof is then completed.
\end{proof}

\textcolor{black}{Finally, we close the section with the following remark:
\begin{Rmk}\label{rem:maxent_maxlikelihood_comaprison}
From the proof of Theorem \ref{theo:ctrlConstr}, we get that,  at time-step $k$, $\left( \tilde{f}_\mathbf{U}^k \right)^\ast$  is obtained as
\begin{equation*}
	\begin{array}{lcll} 
	{ \left( \tilde{f}_\mathbf{U}^k \right)^\ast
	 \in \arg } 
	& \underset{{\tilde{f}_\mathbf{U}^k}}{\min} 
		& \mathcal{D}_{\KL}\left( \tilde{f}^k_{\mathbf{U}}||\tilde{g}^k_{\mathbf{U}} \right)+\E_{ \tilde{f}^k_{\mathbf{U}}}\left[\hat{\omega}\left(\mathbf{U}_k,\mathbf{X}_{k-1}\right)\right]\\
	& { \textnormal{s.t.:} }  	
		& c_{\mathbf{u},j}^k \left[ \tilde{f}_\mathbf{U}^k  \right]= 0, 	
			\ \  {\forall j \in\mathcal{E}_0^k}, \\
	&	
		& c_{\mathbf{u},j}^k \left[ \tilde{f}_\mathbf{U}^k  \right]\leq 0,  
			\ \ { \forall j \in\mathcal{I}^k}\\	
	\end{array} 
\end{equation*}
where $\hat{\omega}\left( \cdot,\,\cdot \right)$ is defined as in Theorem \ref{theo:ctrlConstr}. If for the above problem: (i) $ \hat{\omega} \left( \cdot,\,\cdot \right)$ is constant in the first argument; (ii) $\tilde{g}_\mathbf{U}^k$ is a uniform distribution on a compact support. Then,  its minimizer can also be  found via the  constrained maximum entropy problem:
\begin{equation*}
	\begin{array}{lcll} 
	{ \left( \tilde{f}_\mathbf{U}^k \right)^\ast
	 \in \arg } 
	& \underset{{\tilde{f}_\mathbf{U}^k}}{\max} 
		&  -\E_{\tilde{f}_\mathbf{U}^k}\left[\ln \tilde{f}_\mathbf{U}^k\right] \\
	& { \textnormal{s.t.:} }  	
		& c_{\mathbf{u},j}^k \left[ \tilde{f}_\mathbf{U}^k  \right]= 0, 	
			\ \  {\forall j \in\mathcal{E}_0^k},\\
	&	
		& c_{\mathbf{u},j}^k \left[ \tilde{f}_\mathbf{U}^k  \right]\leq 0,  
			\ \ { \forall j \in\mathcal{I}^k}\\	
	\end{array} 
\end{equation*}
\end{Rmk}}

\section{An algorithm from Theorem \ref{theo:ctrlConstr} }\label{sec:algorithm}
{Theorem \ref{theo:ctrlConstr} gives an explicit expression for synthesizing the control policy and we leverage this to turn the result into an algorithmic procedure. The algorithm introduced here takes as input $\Ipdf \left( \textcolor{black}{\boldsymbol{\Delta}}_e^n \right)$, $\tilde f_{\mathbf{X}}^k:=f(\boldsymbol{x}_k|\boldsymbol{u}_k,\boldsymbol{x}_{k-1})$ and the constraints of Problem \ref{prob:Main_Constr_Ctrl} (if any). The pdfs, as further illustrated in the next section, can be obtained from the \textcolor{black}{data}. Given this input, the algorithm outputs  
$\left\lbrace \left( \tilde{f}_\mathbf{U}^k \right)^\ast \right\rbrace_{k \in \mathcal{K}}$ 
solving Problem \ref{prob:Main_Constr_Ctrl}. \textcolor{black}{Then,  at each $k$, the control input,  $\bf{u}_k$,  applied to the system is obtained by sampling from $\left( \tilde{f}_\mathbf{U}^k \right)^\ast$.} The key steps of the proposed algorithm are summarized as pseudo-code in Algorithm \ref{algo:FPD_Analytical_pseudocode}.
\begin{algorithm}[H]
	\caption{Pseudo-code }
	\label{algo:FPD_Analytical_pseudocode}
	\begin{algorithmic}	
			\State \textbf{Inputs:} 
		\State $\Ipdf \left( \textcolor{black}{\boldsymbol{\Delta}}_e^n \right)$, $\tilde f_{\mathbf{X}}^k$ and constraints of Problem \ref{prob:Main_Constr_Ctrl}
		\State \textbf{Output:}
		\State $\left \lbrace
	\left( \tilde{f}_\mathbf{U}^k \right)^*
	\right  \rbrace_{k \in \mathcal{K}}$ solving Problem \ref{prob:Main_Constr_Ctrl}
		\State \textbf{Initialize} 
		\State {$\hat{\gamma}_{\mathbf{u},{j}}^{n+1} \left( \mathbf{x}_{\Nh} \right) = 1$, $\forall j $};
		\State {$ 
	\hat{\gamma}
	\left(
	\mathbf{x}_{n}
	\right)
	\gets \exp\left[
	\sum_j	\ln
	\left(
	\hat{\gamma}_{\mathbf{u},j}^{n+1}				
	\left(
	\mathbf{x}_{n}
	\right)
	\right) 	
	\right]$};
		\For{$ k = \Nh$  to $1$}		 
		\State
			$\hat{\alpha} \left( \mathbf{u}_{k},	\mathbf{x}_{k-1} \right) 
		 	\gets \int \textcolor{black}{\tilde f_{\mathbf{X}}^k}
		 	{\ln}\frac{   \textcolor{black}{\tilde f_{\mathbf{X}}^k} 
		 	}{  \textcolor{black}{\tilde g_{\mathbf{X}}^k}
		 	} 
		  	\,d\mathbf{x}_{k}
			$;
		
		\State
			$ \hat{\beta} \left( \mathbf{u}_{k},	\mathbf{x}_{k-1} \right)
		 	\gets - \int 
			 \textcolor{black}{\tilde f_{\mathbf{X}}^k}
			\ln 
		 		\left(	 \hat{\gamma} \left( \mathbf{x}_{k} \right) 	\right)
			 \textcolor{black}{d\mathbf{x}_{k}}$;
			
		\State
			$ \hat{\omega} \left( \mathbf{u}_{k},	\mathbf{x}_{k-1} \right)
				\gets 	\hat{\alpha} \left( \mathbf{u}_{k},	\mathbf{x}_{k-1} \right) +
					\hat{\beta} \left( \mathbf{u}_{k},	\mathbf{x}_{k-1} \right) $;
			\State { Compute $\left[ \left(\lambda_{\mathbf{u},1}^k\right)^\ast,\ldots,\left(\lambda_{\mathbf{u},n_e^k+n_l^k}^k\right)^\ast\right]^T$
			by solving \eqref{eq:probl_dual_k_Thm_1} and  $\left(\lambda_{\mathbf{u},0}^k\right)^\ast$ using \eqref{eqn:lambda0_thm}; }	
			\State Compute the control policy:
			\begin{equation*}
			\begin{array}{l}
			\left( \tilde{f}_\mathbf{U}^k \right)^\ast\gets
			\tilde{\Ipdf}_\mathbf{U}^k 
			\frac{e^{-\lbrace 
					\hat{\omega} \left( \mathbf{u}_k,\,\mathbf{x}_{k-1} \right) +
					\sum_{j \in \textcolor{black}{\mathcal{I}_a}^k  \smallsetminus \lbrace 0 \rbrace }
					\left( \lambda_{\mathbf{u},j}^k \right)^\ast
					h_{\mathbf{u},j}^k \left( \mathbf{u}_k \right)	
					\rbrace}}{e^{1+ \left( \lambda_{\mathbf{u},0}^k \right)^* } };
			\end{array}
			\end{equation*} 
			\State Prepare variables for the next iteration, $k-1$:	
			\State 
			$\hat{\gamma}_{\mathbf{u},0}^k 
			\left(\mathbf{x}_{k-1} \right) 
			\gets 
			\exp{ \lbrace \left( \lambda_{\mathbf{u},0}^k \right)^* + 1  \rbrace}$ 
			\State				
			$\hat{\gamma}_{\mathbf{u},j}^k				
			\left(\mathbf{x}_{k-1} \right) \gets  
			\exp{ \lbrace 
					\left( \lambda_{\mathbf{u},j}^k \right)^\ast
				 	H_{\mathbf{u},j}^k	\rbrace}  	
			 		\quad j \in \mathcal{E}^k \cup  \mathcal{I}^k$ 
			\State
				{$ 
	\hat{\gamma}
	\left(
	\mathbf{x}_{k-1}
	\right)
	\gets \exp\left[
	\sum_{j \in \textcolor{black}{\mathcal{I}_a}^k }
	\ln
	\left(
	\hat{\gamma}_{\mathbf{u},j}^k				
	\left(
	\mathbf{x}_{k-1}
	\right)
	\right) 	
	\right]$}
	\EndFor
	\end{algorithmic}
\end{algorithm}
Algorithm \ref{algo:FPD_Analytical_pseudocode} is used to compute the control policy for the example of the next section.

\section{Numerical example}\label{sec:example}
We now illustrate the effectiveness of the results via a numerical example where  Algorithm  \ref{algo:FPD_Analytical_pseudocode} is used to synthesize, from data measured during test drives, a policy for the merging of a car on a highway.  \textcolor{black}{Specifically,  we leverage our results to synthesize a policy that makes the behavior of the controlled car similar to the behavior seen in the examples (collected through test drives) while satisfying some desired constraints on the control variable.} We first describe the scenario considered and the experimental set-up. Then, we describe the data we collected and the process we used to compute the pdfs for the algorithm. Finally, we discuss the results.

\noindent {\bf Scenario description and experimental set-up.} The scenario we considered is schematically illustrated in Figure \ref{fig:UseCase_Scketch_JA}, where a car is merging onto a highway. 
\begin{figure}[htbp]
	\centering	
	\includegraphics[width=0.7\columnwidth]{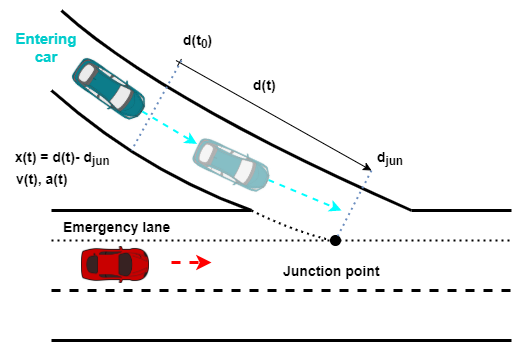}
	\caption{The scenario considered in Section \ref{sec:example}: the light-blue vehicle is trying to merge on a highway.}
	\label{fig:UseCase_Scketch_JA}
\end{figure}
The stretch of road where our experiments took place is outside University College Dublin (UCD)  and the highway is {\em Stillorgan Road} in Dublin $4$ (see Figure \ref{fig:UCMergingStillorgan}). Data were collected from a Toyota Prius using an OBD2 connection with a smartphone running the  Android apps {\em Hybrid Assistant} and its reporting tool {\em Hybrid Reporter}\footnote{See http://hybridassistant.blogspot.com/}. We collected the car GPS position and its longitudinal speed using the hardware-in-the-loop architecture of \citep{griggs2019vehicle}. The apps on the smartphone provided \textit{raw data} in a \texttt{.txt} file with each line reporting the quantities measured at a given time instant (the sampling period was of approximately $0.25$s).  Data were imported in Matlab and the car position was localized by cross-referencing these data with the road information from \textit{OpenStreetMap} \citep{OpenStreetMap2019}.
\begin{figure}[htbp]
	\centering		
	\includegraphics[width=0.9\columnwidth]{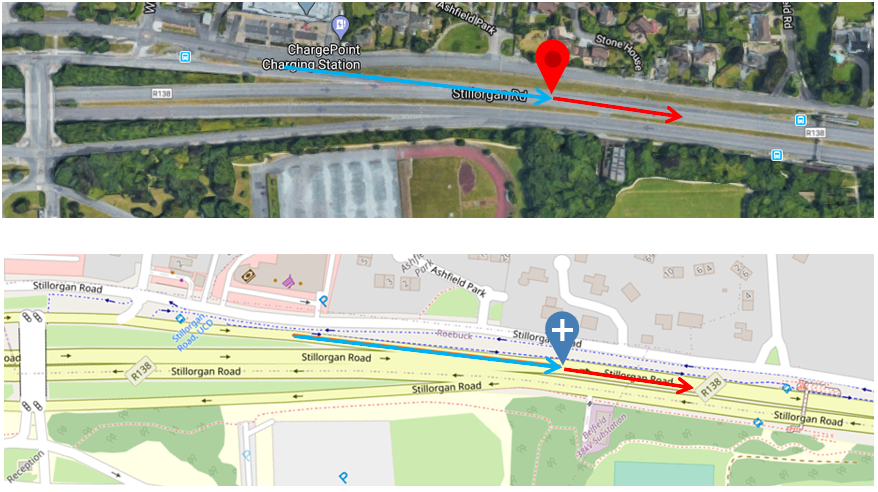}
	\caption{Area for the experiments: map view (from Google maps) and \textit{OpenStreetMap} representation.}
	\label{fig:UCMergingStillorgan}
\end{figure}

\noindent {\bf Collecting the data.}  We performed $100$ test drives. In each of the tests, data were collected within a $300$ meters observation window starting $200 \, m$ before the junction (i.e. after the UCD entrance). The  raw data  were processed to obtain the speed, acceleration and jerk profiles of the car as a function of the distance traveled within the observation window (see left panels in Figure \ref{fig:data}).  \textcolor{black}{Following the notation introduced in Section \ref{ss:Assumpt} (see also Remark \ref{rem:dataset}) the sequence of data collected from each test drive is a dataset. For notational convenience, we simply term the collection of these $100$} datasets  as \textit{complete} dataset in what follows. The vertical line in each panel highlights the physical location of the junction of Figure \ref{fig:UCMergingStillorgan}. From the complete dataset we extracted a subset of profiles that would serve as examples. In particular, we selected  the profiles with the lowest root mean square (RMS) value for the jerk, which is typically associated to a comfortable driving style, see e.g. \citep{Bae_Il_+al_2019_Electronics_Toward_comfortable_self_driving}. We used as example driving profiles, those having a RMS value for the jerk of {at most $0.16$}. This gave the subset of $20$ driving profiles \textcolor{black}{(i.e. a collection of $20$ datasets)} shown in the right panels of Figure \ref{fig:data}.  \textcolor{black}{We simply term the collection of these $20$ datasets as example dataset.}

\noindent {\bf Computing the pdfs.} We let $\mathbf{x}_k$ be the position of the car at the $k$-th time-step (i.e. $\mathbf{x}_k = {d_k}$) and $\mathbf{u}_k$ be its longitudinal speed (i.e. $\mathbf{u}_k =  v_k$).  \textcolor{black}{Following  \citep{8814935},} we computed the joint probability density functions as \textcolor{black}{empirical} distributions to obtain $f(\mathbf{x}_{k-1},\mathbf{u}_k)$ and $g(\mathbf{x}_{k-1},\mathbf{u}_k)$ from the complete and the example dataset respectively (see Figure \ref{fig:scatter_Closed_loop}). \textcolor{black}{Following the same process,} we also obtained $f(\mathbf{x}_{k},\mathbf{x}_{k-1},\mathbf{u}_k)$ and $g(\mathbf{x}_{k},\mathbf{x}_{k-1},\mathbf{u}_k)$ and conditioned these joint pdfs to get  $f(\mathbf{x}_k|\mathbf{x}_{k-1},\mathbf{u}_k) = f(\mathbf{x}_{k},\mathbf{x}_{k-1},\mathbf{u}_k) /f(\mathbf{x}_{k-1},\mathbf{u}_k)$ and $g(\mathbf{x}_k|\mathbf{x}_{k-1},\mathbf{u}_k)= g(\mathbf{x}_{k},\mathbf{x}_{k-1},\mathbf{u}_k) /g(\mathbf{x}_{k-1},\mathbf{u}_k)$. We then assumed $\tilde{f}_{\mathbf{X}}^k$ and  $\tilde{g}_{\mathbf{X}}^k$, two inputs to Algorithm \ref{algo:FPD_Analytical_pseudocode}, to be normal distributions \textcolor{black}{and estimated their parameters via least squares}.
This yielded $\tilde{f}_\mathbf{X}^{k} \sim \mathcal{N} \left(a_c \mathbf{x}_{k-1} + b_c \mathbf{u}_k, \sigma^2_c \right)$, $\tilde{g}_\mathbf{X}^{k} \sim \mathcal{N} 
\left(a_e \mathbf{x}_{k-1} + b_e \mathbf{u}_k, \sigma^2_e \right)$ with $a_c=0.9820$, $b_c = 0.2591 \, s$, $\sigma^2_c = 2.6118 \, m^2$ and $a_e= 0.9811$, $b_e = 0.2723 \, s$, $\sigma^2_e = 1.7622 \, m^2$. Finally, $\tilde{g}_{\mathbf{U}}^k$ was obtained \textcolor{black}{from the empirical pdfs as}
$ \tilde{g}_{\mathbf{U}}^k=  {g\left(\mathbf{x}_{k-1},\mathbf{u}_{k} \right)}/({\int g\left(\mathbf{x}_{k-1},\mathbf{u}_k  \right)d\mathbf{u}_k}).
$}
 		\textcolor{black}{\begin{Rmk}
 		The Gaussian assumption for $\tilde{f}_{\mathbf{X}}^k$ and  $\tilde{g}_{\mathbf{X}}^k$  is inspired by \citep{8287308,moser2015short}. These works considered the problem of making short term predictions of vehicles' trajectories.  See also references therein and our concluding remarks in Section \ref{sec:conclusions}.
 		\end{Rmk}}
\begin{figure}[H]
	\centering	
	\includegraphics[width=0.9\columnwidth, trim={0 0 0 0.77cm}, clip]{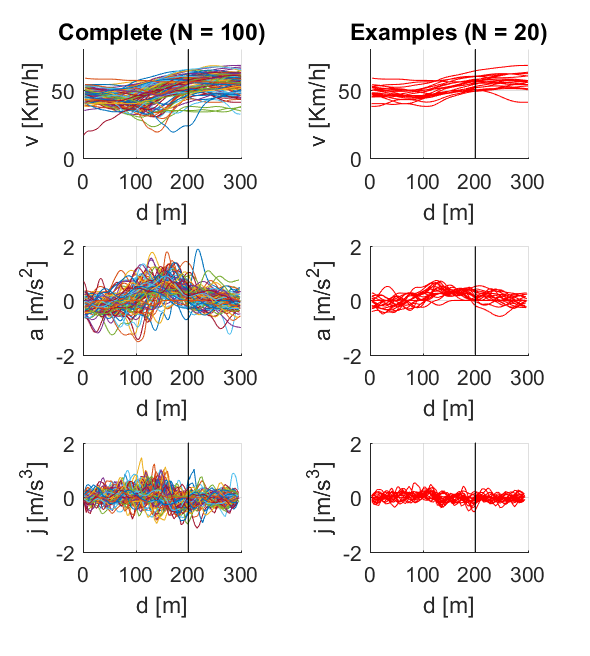}
	\caption{Left panels: driving profiles from the complete dataset of $100$ trips. Right panels: the subset of $20$ profiles used as examples. Acceleration and jerk  computed from the data.}
	\label{fig:data}
\end{figure}

\begin{figure}[htbp]
	\centering	
	\includegraphics[width=0.75\columnwidth]{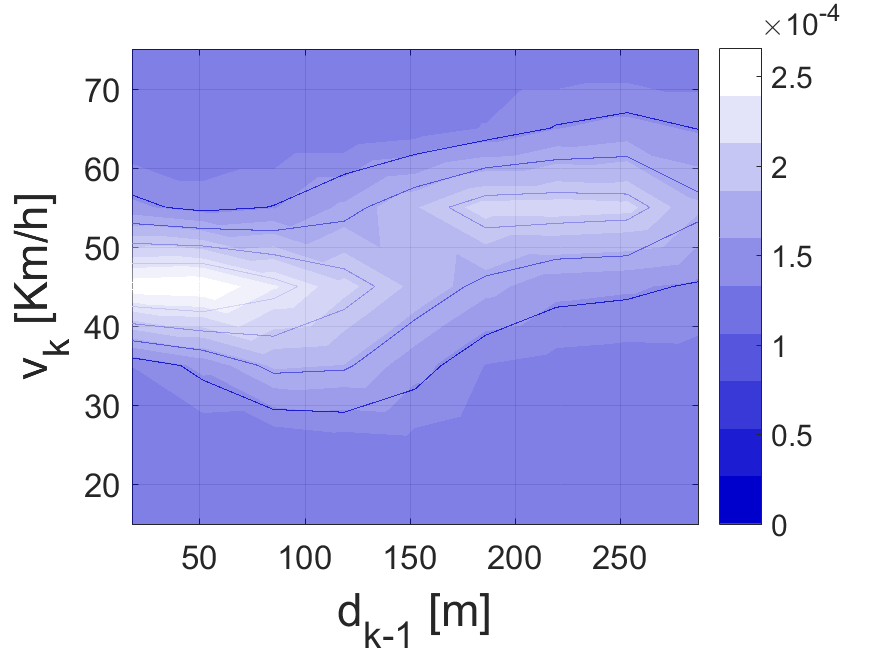}
	\includegraphics[width=0.75\columnwidth]{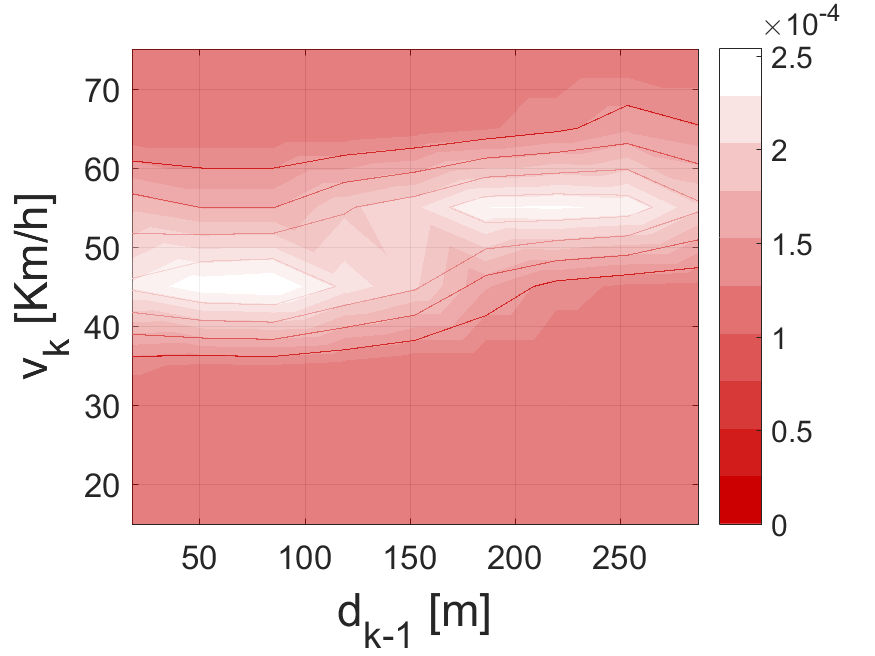}
	\caption{{\em Heat-maps} for  $f(\mathbf{x}_{k-1},\mathbf{u}_k)$ (top panel) and $g(\mathbf{x}_{k-1},\mathbf{u}_k)$ (bottom panel).}
	\label{fig:scatter_Closed_loop}
\end{figure} 

\noindent{\bf Synthesis of the control policy.} Given this set-up, we used Algorithm \ref{algo:FPD_Analytical_pseudocode} to solve Problem \ref{prob:Main_Constr_Ctrl} and hence to synthesize from the examples a control policy allowing the car to merge on the highway.  When synthesizing the control policy,  \textcolor{black}{we imposed, at each $k$, the following set of constraints:
\begin{subequations}\label{eqn:example_constraints}
  \begin{align}
&c_{\mathbf{u},j}^k \left[ \tilde{f}_\mathbf{U}^k  \right] = 0,  \ \ j\in\{0,1,2\},
\\
 \intertext{where:}
&   c_{\mathbf{u},0}^k\left[\tilde{f}_\mathbf{U}^k\right]  = \E_{\tilde{f}_\mathbf{U}^k }
\left[\mathds{1}_{\mathcal{U}_k } (\mathbf{U}_k)\right] - H_{\mathbf{u},0}^k, \\ 
& c_{\mathbf{u},1}^k\left[\tilde{f}_\mathbf{U}^k\right]  = \E_{\tilde{f}_\mathbf{U}^k }
\left[U_k\right] - H_{\mathbf{u},1}^k, \\ 
& c_{\mathbf{u},2}^k\left[\tilde{f}_\mathbf{U}^k\right]  = \E_{\tilde{f}_\mathbf{U}^k }
\left[U_k^2\right] - H_{\mathbf{u},2}^k, \\  
   \intertext{with:}
   & H_{\mathbf{u},0}^k  =1,\\
   &  H_{\mathbf{u},1}^k  =\E_{\tilde{g}_\mathbf{U}^k }
\left[U_k\right],\\
& H_{\mathbf{u},2}^k =4\left(\E_{\tilde{g}_\mathbf{U}^k }
\left[U_k^2\right]-\E_{\tilde{g}_\mathbf{U}^k }^2
\left[U_k\right]\right)+\E_{\tilde{g}_\mathbf{U}^k }^2
\left[U_k\right].
     \end{align}
\end{subequations}
At each $k$, the fulfillment of the first constraint, corresponding to $j=0$, guarantees that the solution to the problem is a pdf (this is  the normalization constraint).  Instead, the fulfillment of the other two constraints in (\ref{eqn:example_constraints}) guarantees that, at each $k$: (i) the expected value of the control variable of the closed loop system is the same as the one seen in the example dataset (constraint corresponding to $j=1$); and (ii) the variance of the control variable of the closed loop system is $4$ times the variance of the control variable seen in the example dataset (constraint corresponding to $j=2$).  Making the variance of the control variable of the closed loop system larger than the control variable from the example dataset  corresponds in accounting for a reduced liability on the example dataset, allowing the closed loop system to depart from the behavior seen in the examples. } \textcolor{black}{Moreover, we note that the constraints in (\ref{eqn:example_constraints}) satisfy  Assumption \ref{asm:constrraints}. Indeed any pdf with the first two moments satisfying the constraints fulfills the assumption.} We used Algorithm \ref{algo:FPD_Analytical_pseudocode} to compute the solution to Problem \ref{prob:Main_Constr_Ctrl} and approximated these pdfs via the Maximum Entropy Principle. This resulted into the truncated Gaussians of Figure \ref{fig:control}. In such a figure, it is clearly shown how these pdfs have higher variance than the corresponding $\tilde{g}_{\mathbf{U}}^k$ extracted from the examples. In the figure, for clarity,  $\tilde{g}_{\mathbf{U}}^k$ is also represented as a truncated Gaussian that has the same mean and variance as the pdf extracted from the data. At each time-step, the control input,  $\mathbf{u}_k$, applied to the car was  obtained by sampling from the pdfs of Figure \ref{fig:control} (in green). In particular,  by sampling the $\mathbf{u}_k$'s as the mean value of the random variable generated by these pdfs, we obtained the speed profile for the controlled car illustrated in Figure \ref{fig:CL_control}.

\begin{figure}[htbp]
	\centering	
	\includegraphics[width=0.9\columnwidth]{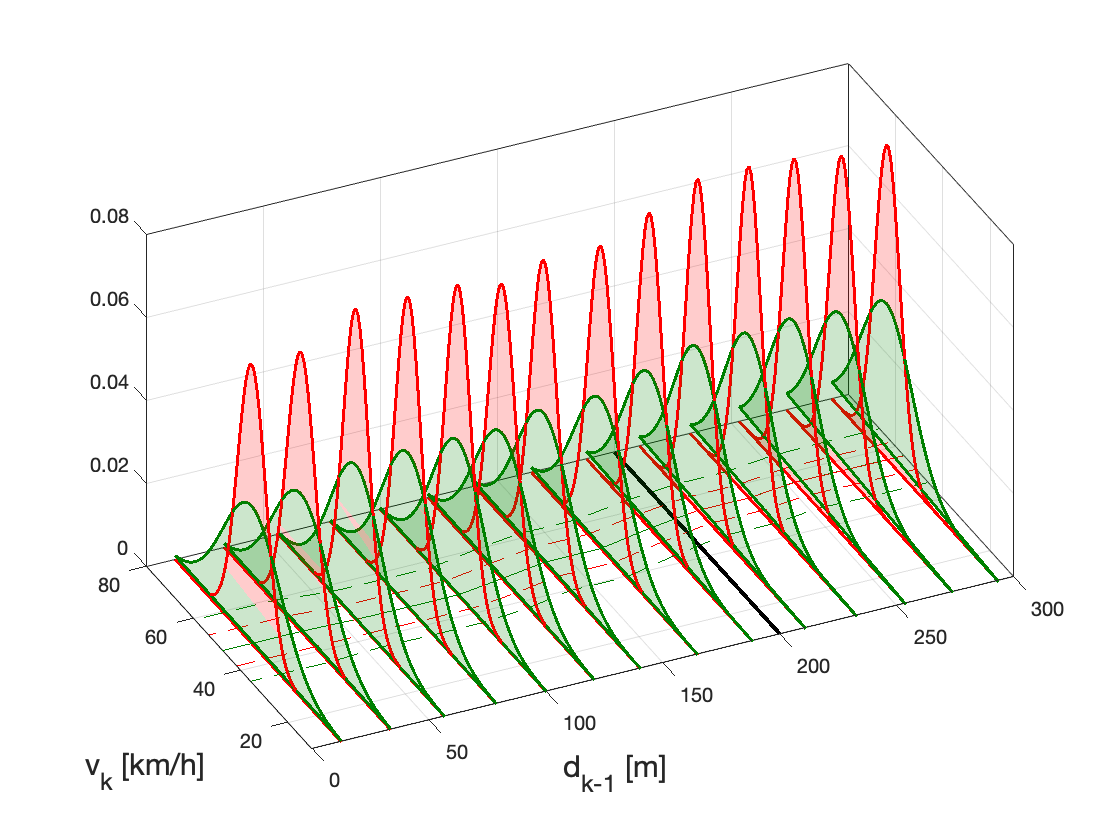}
	\caption{{The pdfs $\tilde{g}_{\mathbf{U}}^k$  (in red) together with the  pdfs obtained from Algorithm \ref{algo:FPD_Analytical_pseudocode} (green). For the sake of clarity in the figure, the pdfs are not shown for each iteration (the policies shown here are representative for all the other time-steps). The continuous line on the speed/distance plane denotes the expectation of the pdfs/policies, while the dashed lines represent {the confidence interval corresponding to the standard deviation} from the examples (red) and from Algorithm \ref{algo:FPD_Analytical_pseudocode} (green). These are shown for each time-step. Colors online.}}
	\label{fig:control}
\end{figure}

\begin{figure}[htbp]
	\centering	
	\includegraphics[width=0.9\columnwidth]{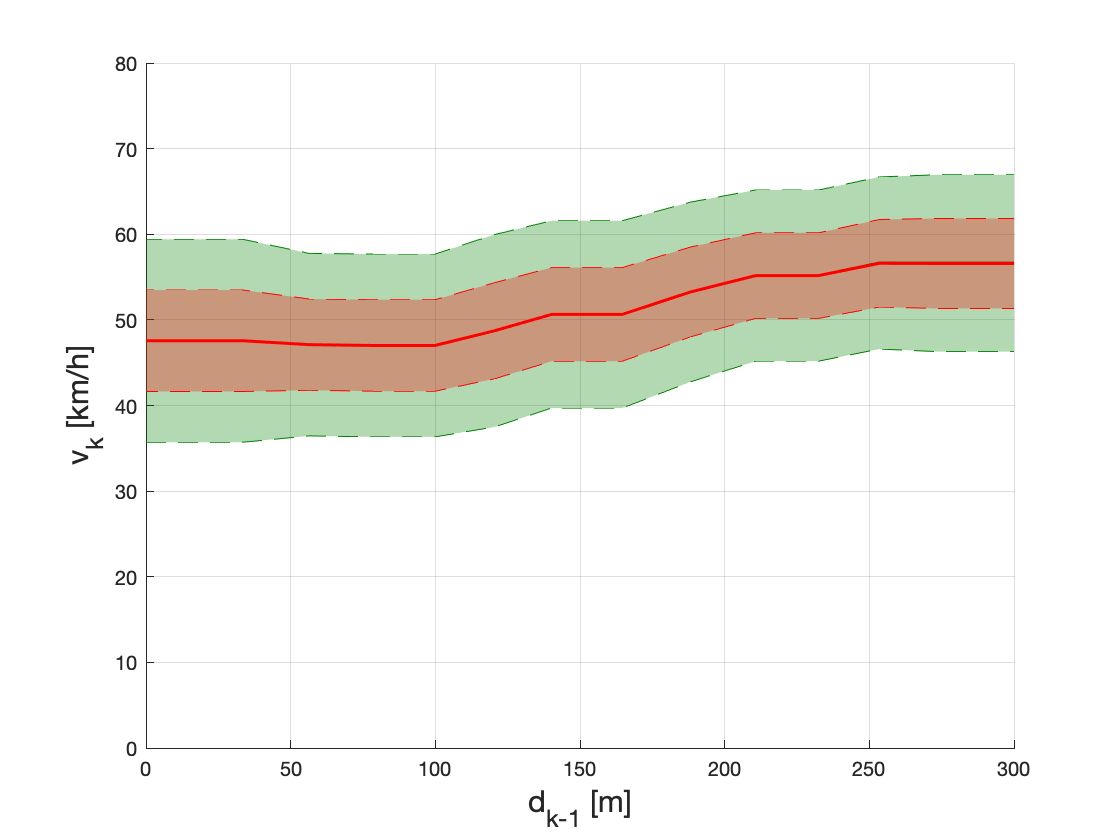}
	\caption{{In green: speed profile from $\left( \tilde{f}_\mathbf{U}^k \right)^\ast$. The bold line is the average speed profile and the shaded areas represent the confidence intervals corresponding to the standard deviation.  \textcolor{black}{For comparison, the corresponding speed profile from the example dataset is also shown in red.  As expected, the bold red line overlaps with the bold green line.  Colors online.}}
	}
	\label{fig:CL_control}
\end{figure}

\section{Conclusions}\label{sec:conclusions}
{We considered the problem of synthesizing control policies from noisy example datasets for systems affected by actuation constraints. To tackle this problem, we introduced a number of technical results to explicitly compute the policy directly from certain pdfs obtained from the data and in compliance with the constraints.  The optimal policy obtained with our results allows to approximate the behavior seen in the examples, while simultaneously fulfilling the system-specific actuation \textcolor{black}{constraints}. The results were also turned into an algorithmic procedure and their effectiveness was illustrated via a use-case. The use-case involved the synthesis, from measured data collected during test drives, of a control policy allowing a car to merge on a highway. \textcolor{black}{We are currently investigating whether our methodology can be extended to consider other divergences \citep{BASSEVILLE2013621} rather than the KL-divergence in the cost  of Problem \ref{prob:Main_Constr_Ctrl}}. Our future work will be aimed at extending the results presented in this paper by considering: (i) the introduction of  chance constraints on the state variable; \textcolor{black}{(ii) cost functionals that do not only aim at tracking the behavior in the examples but also minimize additional costs}; \textcolor{black}{(iii) the use of concepts from data informativity and optimal experimental design  to obtain sufficiently {\em informative} data for the framework developed here. Finally,  we will explore the possibility of devising {\em automated} fitting procedures to extract suitable pdfs from the data in order to enable an end-to-end pipeline for our results. }

\section*{Acknowledgments}
The authors wish to acknowledge Dr. Herzallah (Aston University) for her comments on an earlier version of this manuscript during her visit at UCD, Dr. Guy and Prof. K{\'a}rn{\'y} (both at the Institute of Information Theory and Automation at the  Czech Academy of Sciences) for the insightful discussions on the results.  GR would also like to thank Prof. Bullo at UCSB for reading an earlier version of this work.  {Five anonymous referees made several helpful comments and suggestions, which led to improvements over the originally submitted manuscript.}
\appendix
\section{Appendix: proofs of the technical results}
\subsection{Proof of Property \ref{proper:KLsplit}}
To prove this result we start from the definition of \KL-divergence. In particular:
\begin{equation*}
	\begin{split}
&		\mathcal{D}_{\KL}
		\left( 
		\phi \left( \mathbf{y},\mathbf{z} \right) || 
		g\left( \mathbf{y},\mathbf{z} \right) 
		\right)  := 	\int
		\int
		\phi \left( \mathbf{y},\mathbf{z} \right)
		\left[
		\ln \frac{ \phi \left( \mathbf{y},\mathbf{z} \right)}
		{g\left( \mathbf{y},\mathbf{z} \right) } 
		\right]
		\,d\mathbf{y}
		\,d\mathbf{z}	 \\
		& =
		\underbrace{	
			\int
			\int
			\phi \left( \mathbf{z}|\mathbf{y} \right)
			\left[
			\phi \left( \mathbf{y}\right) 
			\ln 
			\frac{
				\phi \left( \mathbf{y}\right) 
			}{
				g \left( \mathbf{y}\right) 
			} 
			\right]
			\, d\mathbf{y}
			\, d\mathbf{z}
		}_{(1)}	\;	 \\
		& 	+
		\underbrace{
			\int 
			\int
			\phi \left( \mathbf{y}\right) 
			\left[
			\phi \left( \mathbf{z}|\mathbf{y} \right)	
			\ln \frac{ 	\phi \left( \mathbf{z}|\mathbf{y} \right)
			}{
				g \left( \mathbf{z}|\mathbf{y} \right)} 
			\right] 
			\, d\mathbf{z}
			\, d\mathbf{y}
		}_{(2)}.	
	\end{split}
\end{equation*}
For the term $(1)$ in the above expression we may continue as follows: 
$		\int
		\int
		\phi \left( \mathbf{z}|\mathbf{y} \right)
		\left[
		\phi \left( \mathbf{y}\right) 
		\ln 
		\frac{
			\phi \left( \mathbf{y}\right) 
		}{
			g \left( \mathbf{y}\right) 
		} 
		\right]
		\, d\mathbf{y}
		\, d\mathbf{z}  = 
		\int
		\phi \left( \mathbf{z}|\mathbf{y} \right)
		\, d\mathbf{z}
		\left[\int
		\phi \left( \mathbf{y}\right) 
		\ln 
		\frac{
			\phi \left( \mathbf{y}\right) 
		}{
			g \left( \mathbf{y}\right) 
		} 
		d\mathbf{y}\right] =
		\mathcal{D}_{\KL} \left(\phi \left( \mathbf{y}\right)  || g \left( \mathbf{y}\right)  \right)$,
where we used Fubini's theorem, the fact that the term on the first line in square brackets is \textcolor{black}{independent} on $\mathbf{Z}$ and the fact that $\int
		\phi \left( \mathbf{z}|\mathbf{y} \right)
		\, d\mathbf{z} =1$. By using again Fubini's theorem, for the term $(2)$ in the above expression instead we have:
$		\int 
		\int
		\phi \left( \mathbf{y}\right) 
		\left[
		\phi \left( \mathbf{z}|\mathbf{y} \right)	
		\ln \frac{ 	\phi \left( \mathbf{z}|\mathbf{y} \right)
		}{
			g \left( \mathbf{z}|\mathbf{y} \right)} 
		\right] 
		\, d\mathbf{z}
		\, d\mathbf{y}
= \mathbb{E}_{\phi}
		\left[	
		\mathcal{D}_{\KL} \left( 	\phi(\mathbf{z}|\mathbf{Y}) || \, g(\mathbf{z}|\mathbf{Y}) \right)
		\right]$,
thus proving the result.
\qed

\subsection{Proof of Lemma \ref{lem:Constrained_KL}}
	
The proof is organized in $3$ steps. 
In {\bf Step 1} we show that the optimization problem in \eqref{eq:probl_Lemma_1} is a convex optimization problem (COP) and we then devise its augmented Lagrangian.  In {\bf Step 2} we explicit the Karush-Kuhn-Tucker (KKT) conditions and verify that these are satisfied by the solution in (\ref{eq:Lem1_opt_fu_sol}). Recall that, for a COP, KKT conditions are \textcolor{black}{necessary} and sufficient\citep[Chapter $5$]{Boyd_S_Vanderberge_L_2004_Book_Convex_Optim}.
Finally, in {\bf Step 3}, we compute the minimum of the cost function corresponding to the optimal solution.

{\bf Step 1.} We start with observing that the cost function $\mathcal{L}\left(\textcolor{black}{f(\mathbf{z})}\right)$ in (\ref{eq:probl_Lemma_1}) can be conveniently rewritten as $\mathcal{L}\left(\textcolor{black}{f(\mathbf{z})}\right) = \int l\left(\textcolor{black}{f(\mathbf{z})}\right) \,d\mathbf{z}$,
with $	l\left(\textcolor{black}{f(\mathbf{z})}\right) := \textcolor{black}{f(\mathbf{z})} \; \left[ \ln \left( \frac{\textcolor{black}{f(\mathbf{z})}}{\textcolor{black}{g(\mathbf{z})}} \right) + \alpha\textcolor{black}{\left(\mathbf{z}\right)} \right]$.
Clearly, $\mathcal{L}(\cdot)$ is twice differentiable and we now prove that it is also a {strictly} convex functional in $f$. We do this by showing that its second variation is positive definite on the space of integrable functions and we explicit the first and the second variation of $\mathcal{L}\left(\textcolor{black}{f(\mathbf{z})}\right)$, i.e. $\delta\mathcal{L}(f,\delta f)$ and $\delta^2 \mathcal{L}(f,\delta f)$, in terms of the first and second derivative of $l\left(\textcolor{black}{f(\mathbf{z})}\right)$ with respect to $\textcolor{black}{f(\mathbf{z})}$
(i.e. $\frac{\partial l\left(\textcolor{black}{f(\mathbf{z})}\right)}{\partial f}$ and $\frac{\partial^2 l\left(\textcolor{black}{f(\mathbf{z})}\right)}{\partial f^2}$, respectively). By computing $\delta\mathcal{L}(f,\delta f)$ we get $\delta \mathcal{L} \left( f,\delta f \right) 
=  \int  \frac{\partial l \left(\textcolor{black}{f(\mathbf{z})}\right)}{\partial f}  \delta f \, d\mathbf{z}$,
with $\frac{\partial l\left(\textcolor{black}{f(\mathbf{z})}\right)}{\partial f}  = \ln \textcolor{black}{f(\mathbf{z})} + \left( \alpha(\mathbf{z}) + 1 -\ln \textcolor{black}{g(\mathbf{z})} \right)$. This leads to the following expression for the second variation of $\mathcal{L}(f,\delta f)$ 
{\begin{equation}\label{eqn:secondvar}
\delta^2 \mathcal{L} \left( f,\delta f \right) 
= \int  \delta f \, \left( \frac{\partial^2 l\left(\textcolor{black}{f(\mathbf{z})}\right)}{\partial f^2} \right) \, \delta f \, d\mathbf{z}
\end{equation}
To show convexity of $\mathcal{L}$ it then suffices to observe that, since $f(\mathbf{z})$ is \textcolor{black}{non-negative} on its support, then 
the quantity under the integral in (\ref{eqn:secondvar}) is positive.} 
In turn, this implies  that $\delta^2 \mathcal{L} \left( f,\delta f \right)$ is strictly positive for any measurable, non-zero variation $\delta f$ (see also \citep[Chapter $4$]{Kirk_DE_2004_OptimControl} for a detailed discussion). Hence, in order to conclude that the problem in \eqref{eq:probl_Lemma_1} is a COP, it suffices to observe that the constraints in \eqref{eq:constr_c_j} are linear in $f(\mathbf{z})$. The augmented Lagrangian of the COP in \eqref{eq:probl_Lemma_1} is: 
\begin{equation}
\label{eq:L_aug_std}
\begin{split}
\mathcal{L}_{aug} \left( \textcolor{black}{f(\mathbf{z})}, \boldsymbol{\lambda} \right) := 
	& \mathcal{L} \left( \textcolor{black}{f(\mathbf{z})} \right) 	
	+ \sum_{j \in \mathcal{E}_0 \cup  \mathcal{I}}  
	\lambda_j \, c_j \left[\textcolor{black}{f(\mathbf{z})}\right],	 	
\end{split}
\end{equation}
where $\boldsymbol{\lambda}:= [\lambda_0,\lambda_1,\dots,\lambda_{n_e + n_{\textcolor{black}{{l}}}}]^T$ is the column vector stacking all the  \textcolor{black}{Lagrange multipliers}.

{\bf Step 2.} We showed that the problem in (\ref{eq:probl_Lemma_1}) is a COP and hence the KKT conditions are necessary and sufficient optimality conditions. That is, in order to be a unique minimizer of the problem, the candidate function $f(\mathbf{z})$ must satisfy the conditions made explicit in Table \ref{tab:KKT_cond_stack}.
\begin{table}[t!]
  	\small
  	\begin{center}
 	 \begin{tabular}{lll}    		
      \textit{Primal feasibility}:      	& $c_j \left[f(\mathbf{z})\right] = 0$,  		& $\forall j \in \mathcal{E}_0$	\\
       										& $c_j \left[f(\mathbf{z})\right]	\leq 0$, 	& $\forall j \in \mathcal{I}$	\\
      \hline
      \textit{Dual feasibility}: 			& $\lambda_j \geq 0,$ 	& $\forall j \in \mathcal{I}	$\\
      \hline
      \textit{Complementary slackness}: 	& $\lambda_j \, c_j \left[f(\mathbf{z})\right] = 0,$ & $ \forall j \in \mathcal{I}$\\
      \hline
      \textit{Stationarity}: 				& $\delta \mathcal{L}_{aug} \left( f,\delta f, \boldsymbol{\lambda} \right) = 0,$ & $\forall \,\delta f$\\
    \end{tabular}	
      		\caption{KKT conditions for the problem \textcolor{black}{in} \eqref{eq:probl_Lemma_1}.}
  	    \label{tab:KKT_cond_stack}
\end{center}
\end{table}
We now impose the {stationarity} condition (see Table \ref{tab:KKT_cond_stack}) and first note that the augmented Lagrangian \eqref{eq:L_aug_std} can be written as: $
\mathcal{L}_{aug} \left( \textcolor{black}{f(\mathbf{z})}, \boldsymbol{\lambda} \right) :=  \int \textcolor{black}{f(\mathbf{z})}\; 
\left[
\ln
\left(
\frac{\textcolor{black}{f(\mathbf{z})}}{
\textcolor{black}{g(\mathbf{z})}}
\right)
+ 
\alpha(\mathbf{z})
\right] \,
d\mathbf{z} +
\langle 
\boldsymbol{\lambda},
\int  \textcolor{black}{f(\mathbf{z})} \; 
\mathbf{h}\left( \mathbf{z} \right)
d\mathbf{z} 
-\mathbf{H}
\rangle $.
Hence:
\begin{equation}
\label{eqn:lagrangian_Lema_1}
\begin{array}{ll}
\mathcal{L}_{aug}
\left(  \textcolor{black}{f(\mathbf{z})}, \boldsymbol{\lambda} \right)  =
\int 
\tilde{l} \left(  \textcolor{black}{f(\mathbf{z})},	\boldsymbol{\lambda}\right)
d\mathbf{z}	- \langle\boldsymbol{\lambda},\mathbf{H}\rangle,
\end{array}
\end{equation}
where the quantity under the integral is given by $\tilde{l} \left(  \textcolor{black}{f(\mathbf{z})},	\boldsymbol{\lambda}\right)	:=
 \textcolor{black}{f(\mathbf{z})} \; 
\left[
\ln
\left(
\frac{ \textcolor{black}{f(\mathbf{z})}}{
	 \textcolor{black}{g(\mathbf{z})}}
\right)
+ \alpha(\mathbf{z}) +  \langle \boldsymbol{\lambda}, \mathbf{h}(\mathbf{z}) \rangle
\right]$.
By computing the first variation of \eqref{eqn:lagrangian_Lema_1} with respect to $\delta f$ we obtain
$
\delta \mathcal{L}_{aug} \left( f,\delta f, \boldsymbol{\lambda} \right) 
=  \int  \frac{\partial  \tilde{l}\left( \textcolor{black}{f(\mathbf{z})},\boldsymbol{\lambda}\right) }{\partial f}  \delta f  \, d\mathbf{z}$,
and thus, by imposing  the stationarity  condition (i.e. $\delta \mathcal{L}_{aug} \left( f, \delta f, \boldsymbol{\lambda} \right)  = 0$, $\forall \delta f$), we get $\frac{\partial \tilde{l}( \textcolor{black}{f(\mathbf{z})},\boldsymbol{\lambda})}{\partial f }  = 0$. That is,
$\frac{\partial  \tilde{l}\left( \textcolor{black}{f(\mathbf{z})},\boldsymbol{\lambda}\right)}{\partial f }  
= \ln \left( \frac{ \textcolor{black}{f(\mathbf{z})}}{ \textcolor{black}{g(\mathbf{z})}} \right) + {\alpha}\left(\mathbf{z}\right) +\langle \boldsymbol{\lambda},\mathbf{h}(\mathbf{z})\rangle + 1 = 0$, 
from which it immediately follows that all the optimal solution candidates must be of the form
\begin{equation}
\label{eq:f_Laug_staz_candidate}
\hat{f}^\ast \left(\mathbf{z} ,\boldsymbol{\lambda} \right):= g(\mathbf{z})
	e^{-\lbrace 1+ \alpha(\mathbf{z}) + \langle \boldsymbol{\lambda},\mathbf{h}\left( \mathbf{z} \right) \rangle \rbrace}.
\end{equation}
In the above expression, the notation $\hat{f}^* \left(\mathbf{z} ,\boldsymbol{\lambda} \right)$ was introduced to stress that the optimal solution candidate is a function of the  \textcolor{black}{Lagrange multipliers}.  {These can be found by solving the following dual problem
\begin{equation}
\label{eq:probl_dual_Lemma_1}
\begin{array}{lllll} 
\boldsymbol{\lambda}^\ast \in \arg 
	&  \underset{\boldsymbol{\lambda}}{ \max } 	
		& \mathcal{L}^D \left(\boldsymbol{\lambda}\right) \\	
	&	\text{s.t.:} 								
			&  \lambda_j \text{ free }, 	
				& \forall j \in \mathcal{E}_0,  \\
	& 
		& \lambda_j \geq 0,  		
			& \forall j \in \mathcal{I} \\ 	
\end{array} 
\end{equation}
}
{choosing $\boldsymbol{\lambda}^\ast$ so that $\hat{f}^\ast \left(\mathbf{z} ,\boldsymbol{\lambda}^\ast \right)$ is primal feasible} (see  \citep{Roc_88}). In the  problem, $\mathcal{L}^{D}
\left( \boldsymbol{\lambda} \right)$ is the Lagrange dual function $	\mathcal{L}^{D}
	\left( \boldsymbol{\lambda} \right)
	:=  \underset{\textcolor{black}{f(\mathbf{z})}\geq 0 }{ \inf} 	 \mathcal{L}_{aug}
\left(\textcolor{black}{f(\mathbf{z})}, \boldsymbol{\lambda} \right)$.
Note that the vector $\boldsymbol{\lambda}^\ast$ must satisfy the dual feasibility condition. Now, Assumption \ref{asm:constrraints} implies strong duality (see Remark \ref{rem:duality}) and hence  the complementary slackness condition (see Table \ref{tab:KKT_cond_stack}) is also fulfilled. 
Additionally, $\mathcal{L}_{aug}\left(\textcolor{black}{f(\mathbf{z})}, \boldsymbol{\lambda} \right)$ is strictly convex in $f(\mathbf{z})$ and hence $\underset{f(\mathbf{z})\geq 0 }{ \inf} \mathcal{L}_{aug} \left(\textcolor{black}{f(\mathbf{z})}, \boldsymbol{\lambda} \right) = \mathcal{L}_{aug} \left(\hat{f}^\ast \left(\mathbf{z} ,\boldsymbol{\lambda} \right), \boldsymbol{\lambda} \right)$. 
Thus:
{\begin{equation}
\label{eq:Lagr_dual}
\begin{split}
 \mathcal{L}^{D}
\left( \boldsymbol{\lambda} \right) & =
\mathcal{L}_{aug} 
\left( 
\hat{f}^* \left(\mathbf{z}, \boldsymbol{\lambda} \right)
,\boldsymbol{\lambda}
\right)  = 
-\int 
\hat{f}^* \left(\mathbf{z}, \boldsymbol{\lambda} \right)
d\mathbf{z} 
- \langle\boldsymbol{\lambda},\mathbf{H}\rangle \\
& = 
-\int 
g  \left(\mathbf{z}\right)  \;
e^{-\lbrace 1+ \alpha(\mathbf{z}) + \langle \boldsymbol{\lambda},\mathbf{h}\left( \mathbf{z} \right) \rangle \rbrace}
d\mathbf{z} 
- \langle\boldsymbol{\lambda},\mathbf{H}\rangle.
\end{split}
\end{equation}}
{Note now that the last equivalence gives  \eqref{eq:Lagr_dual_final} in the statement of the lemma and hence the problem in (\ref{eq:probl_dual_Lemma_1_statement}).} {Moreover, the complementary slackness condition on the pair of optimizers $f^\ast(\mathbf{z}),\boldsymbol{\lambda}^\ast$ implies, for a COP, that there is no duality gap. That is, $\mathcal{L}^D\left( \boldsymbol{\lambda}^\ast \right) = \mathcal{L} \left( f^\ast(\mathbf{z}) \right)$.} {In turn, this  means that the  \textcolor{black}{Lagrange multipliers} associated to inactive inequality constraints must be equal to $0$, while all the  \textcolor{black}{Lagrange multipliers} associated to active inequality constraints must be non-negative. } Therefore, the optimal solution of the COP in \eqref{eq:probl_Lemma_1} is given by
$\hat{f}^\ast \left(\mathbf{z} , \boldsymbol{\lambda}^\ast \right) = f^\ast(\mathbf{z})
= 	
g(\mathbf{z}) \, e^{- 
		\left \lbrace
		1+ \alpha \left( \mathbf{z} \right) +
		\sum_{j \in \textcolor{black}{\mathcal{I}_a} \left( \textcolor{black}{f^\ast(\mathbf{z})}  \right) }
		\lambda_j^\ast h_j \left( \mathbf{z} \right)
		\right \rbrace 
	}$,
{which was obtained by taking into account that only the  \textcolor{black}{Lagrange multipliers} associated to the active constraints are non-zero. } {The above expression is equal to \eqref{eq:Lem1_opt_fu_sol} where we highlighted the role of $\lambda_0$ as a normalization constant.
This concludes the proof of {\bf (R1)}.}

{\bf Step 3.} Finally, since there is no duality gap, the minimum value of the primal problem (i.e. the COP in (\ref{eq:probl_Lemma_1})) can be obtained from \eqref{eq:Lagr_dual}. 
This leads to
$\mathcal{L} \left(\textcolor{black}{f^\ast(\mathbf{z})}  \right)  
= - \left( 
1 + \sum_{j \in \textcolor{black}{\mathcal{I}_a} \left( \textcolor{black}{f^\ast(\mathbf{z})}  \right) } 
\lambda^\ast_j \, H_j	
\right)$,
and thus completes the proof. \qed

\end{document}

%% file: Packages_DGA_Pack_4FPD.tex
\usepackage{amssymb}
\usepackage{eqnarray} 

\usepackage{amsmath}
\usepackage{mathalfa} 
\usepackage{graphicx}	
\usepackage{datetime2}
\usepackage{pdfpages}		
\usepackage{dsfont}   		
\usepackage{algorithm}
\usepackage{algpseudocode}

\usepackage{bbm}

%% file: NewCommands_DGA_Pack_4FPD_v2.tex
\newtheorem{Thm}{Theorem}

\newtheorem{Lem}{Lemma}
\newtheorem{Rmk}{Remark}
\newtheorem{Def}{Definition}
\newtheorem{Property}{Property}

\newtheorem{Prob}{Problem}   
\newtheorem{Asm}{Assumption} 
\newenvironment{proof}{\textit{Proof:}}{\hfill$\square$}

\newcommand{\Nh}{n}

\newcommand{\R}{\mathbb{R}}
\newcommand{\E}{\mathbb{E}}

\newcommand{\KL}{\text{KL}}
\newcommand{\Ipdf}{g}

\newcommand{\constrVecSymb}{\mathbf{C}}
\newcommand{\constrVec}[2]{\constrVecSymb_{#1}^{#2}}

\newcommand{\eqlmarg}       {\hspace{3mm}}   

%% file: root_rev_20211112_ARXIV.bbl
\begin{thebibliography}{73}
\expandafter\ifx\csname natexlab\endcsname\relax\def\natexlab#1{#1}\fi
\expandafter\ifx\csname url\endcsname\relax
  \def\url#1{\texttt{#1}}\fi
\expandafter\ifx\csname urlprefix\endcsname\relax\def\urlprefix{URL }\fi

\bibitem[{Abbeel and Ng(2004)}]{10.1145/1015330.1015430}
Abbeel, P., Ng, A.~Y., 2004. Apprenticeship learning via inverse reinforcement
  learning. In: Proceedings of the Twenty-First International Conference on
  Machine Learning. ICML '04. Association for Computing Machinery, New York,
  NY, USA, p.~1.

\bibitem[{Argall et~al.(2009)Argall, Chernova, Veloso, and
  Browning}]{ARGALL2009469}
Argall, B.~D., Chernova, S., Veloso, M., Browning, B., 2009. A survey of robot
  learning from demonstration. Robotics and Autonomous Systems 57~(5), 469 --
  483.

\bibitem[{Bae et~al.(2019)Bae, Moon, and
  Seo}]{Bae_Il_+al_2019_Electronics_Toward_comfortable_self_driving}
Bae, I., Moon, J., Seo, J., 2019. Toward a comfortable driving experience for a
  self-driving shuttle bus. Electronics 8~(9), 943.

\bibitem[{{Baggio} et~al.(2019){Baggio}, {Katewa}, and {Pasqualetti}}]{8703172}
{Baggio}, G., {Katewa}, V., {Pasqualetti}, F., 2019. Data-driven minimum-energy
  controls for linear systems. IEEE Control Systems Letters 3~(3), 589--594.

\bibitem[{Basseville(2013)}]{BASSEVILLE2013621}
Basseville, M., 2013. Divergence measures for statistical data processing—an
  annotated bibliography. Signal Processing 93~(4), 621--633.

\bibitem[{Ben-Tal et~al.(1988)Ben-Tal, Teboulle, and Charnes}]{Roc_88}
Ben-Tal, A., Teboulle, M., Charnes, A., 1988. The role of duality in
  optimization problems involving entropy functionals with applications to
  information theory. Journal of optimization theory and applications 58,
  209--223.

\bibitem[{Bertsekas(2021)}]{9317713}
Bertsekas, D., 2021. Multiagent reinforcement learning: Rollout and policy
  iteration. IEEE/CAA Journal of Automatica Sinica 8~(2), 249--272.

\bibitem[{Bot et~al.(2005)Bot, Grad, and Wanka}]{Bot_05}
Bot, R., Grad, S.-M., Wanka, G., 2005. Duality for optimization problems with
  entropy-like objective functions. Journal of Information and Optimization
  Sciences 22, 415--441.

\bibitem[{Boyd and
  Vandenberghe(2004)}]{Boyd_S_Vanderberge_L_2004_Book_Convex_Optim}
Boyd, S.~P., Vandenberghe, L., 2004. Convex optimization. Cambridge
  {U}niversity {P}ress.

\bibitem[{{Bryson}(1996)}]{506395}
{Bryson}, A.~E., June 1996. Optimal control-1950 to 1985. IEEE Control Systems
  Magazine 16~(3), 26--33.

\bibitem[{Censor and
  Elfving(1982)}]{Censor_Y_&_ELfving_T_JLinAlg&App_1982_new_meth_Lin_Ineq}
Censor, Y., Elfving, T., 1982. New methods for linear inequalities. Linear
  Algebra and Its Applications 42, 199--211.

\bibitem[{Colin et~al.(2020)Colin, Bombois, Bako, and
  Morelli}]{COLIN2020109000}
Colin, K., Bombois, X., Bako, L., Morelli, F., 2020. Data informativity for the
  open-loop identification of mimo systems in the prediction error framework.
  Automatica 117, 109000.

\bibitem[{{Coulson} et~al.(2019{\natexlab{a}}){Coulson}, {Lygeros}, and
  {Dörfler}}]{8795639}
{Coulson}, J., {Lygeros}, J., {Dörfler}, F., 2019{\natexlab{a}}. Data-enabled
  predictive control: In the shallows of the {DeePC}. In: 2019 18th European
  Control Conference (ECC). pp. 307--312.

\bibitem[{{Coulson} et~al.(2019{\natexlab{b}}){Coulson}, {Lygeros}, and
  {Dörfler}}]{9028943}
{Coulson}, J., {Lygeros}, J., {Dörfler}, F., 2019{\natexlab{b}}. Regularized
  and distributionally robust data-enabled predictive control. In: 2019 IEEE
  58th Conference on Decision and Control (CDC). pp. 2696--2701.

\bibitem[{Cover and Thomas(2006)}]{10.5555/1146355}
Cover, T.~M., Thomas, J.~A., 2006. Elements of Information Theory (Wiley Series
  in Telecommunications and Signal Processing). Wiley-Interscience, USA.

\bibitem[{{De Persis} and {Tesi}(2020)}]{8933093}
{De Persis}, C., {Tesi}, P., 2020. Formulas for data-driven control:
  Stabilization, optimality, and robustness. IEEE Transactions on Automatic
  Control 65~(3), 909--924.

\bibitem[{Deng et~al.(2019)Deng, Gagliardi, and Del~Re}]{8814935}
Deng, J., Gagliardi, D., Del~Re, L., 2019. Microscopic driving behavior
  modelling at highway entrances using bayesian network. In: 2019 American
  Control Conference (ACC). pp. 977--982.

\bibitem[{Duffin et~al.(1956)Duffin, Dantzig, and
  Fan}]{Duffin_RJ_Dantzig_GB_Fan8K_Book_1956}
Duffin, R.~J., Dantzig, G.~B., Fan, K., 1956. Linear inequalities and related
  systems. Princeton University Press.

\bibitem[{Edwards et~al.(2019)Edwards, Sahni, Schroecker, and
  Isbell}]{pmlr-v97-edwards19a}
Edwards, A., Sahni, H., Schroecker, Y., Isbell, C., 09--15 Jun 2019. Imitating
  latent policies from observation. In: Proceedings of the 36th International
  Conference on Machine Learning. Vol.~97 of Proceedings of Machine Learning
  Research. PMLR, pp. 1755--1763.

\bibitem[{Englert et~al.(2017)Englert, Vien, and
  Toussaint}]{doi:10.1177/0278364917745980}
Englert, P., Vien, N.~A., Toussaint, M., 2017. Inverse {KKT}: Learning cost
  functions of manipulation tasks from demonstrations. The International
  Journal of Robotics Research 36~(13-14), 1474--1488.

\bibitem[{Fan(1968)}]{Fan_K_JMAnnApp_1968_infinite}
Fan, K., 1968. On infinite systems of linear inequalities. Journal of
  Mathematical Analysis and Applications 21~(3), 475--478.

\bibitem[{Fan(1975)}]{Fan_K_JLinAlg&App_1975_two_Aap_of_consistency}
Fan, K., 1975. Two applications of a consistency theorem for systems of linear
  inequalities. Linear Algebra and its Applications 11~(2), 171--180.

\bibitem[{Gagliardi and
  Russo(2020)}]{Gagliardi_D_et_Russo_G_IFAC2020_extended_Arxiv}
Gagliardi, D., Russo, G., 2020. On the synthesis of control policies from
  example datasets. In: 21st IFAC World Congress (see
  \url{https://arxiv.org/abs/2001.04428}).

\bibitem[{Garrabe and Russo(2022)}]{9446558}
Garrabe, E., Russo, G., 2022. On the design of autonomous agents from multiple
  data sources. IEEE Control Systems Letters 6, 698--703.

\bibitem[{Georgiou and Lindquist(2003)}]{1246014}
Georgiou, T., Lindquist, A., 2003. Kullback-{L}eibler approximation of spectral
  density functions. IEEE Transactions on Information Theory 49~(11),
  2910--2917.

\bibitem[{{Gonçalves da Silva} et~al.(2019){Gonçalves da Silva}, {Bazanella},
  {Lorenzini}, and {Campestrini}}]{8453019}
{Gonçalves da Silva}, G.~R., {Bazanella}, A.~S., {Lorenzini}, C.,
  {Campestrini}, L., 2019. Data-driven {LQR} control design. IEEE Control
  Systems Letters 3~(1), 180--185.

\bibitem[{Griggs et~al.(2019)Griggs, Ord{\'o}{\~n}ez-Hurtado, Russo, and
  Shorten}]{griggs2019vehicle}
Griggs, W., Ord{\'o}{\~n}ez-Hurtado, R., Russo, G., Shorten, R., 2019. A
  vehicle-in-the-loop emulation platform for demonstrating intelligent
  transportation systems. In: Control Strategies for Advanced Driver Assistance
  Systems and Autonomous Driving Functions. Springer, pp. 133--154.

\bibitem[{Guan et~al.(2014)Guan, Raginsky, and Willett}]{6716965}
Guan, P., Raginsky, M., Willett, R.~M., 2014. Online {M}arkov {D}ecision
  processes with {K}ullback–{L}eibler control cost. IEEE Transactions on
  Automatic Control 59~(6), 1423--1438.

\bibitem[{Guy et~al.(2018)Guy, Derakhshan, and
  {\v{S}}t{\v{e}}ch}]{10.1007/978-3-030-01713-2_20}
Guy, T.~V., Derakhshan, S.~F., {\v{S}}t{\v{e}}ch, J., 2018. Lazy fully
  probabilistic design: Application potential. In: Belardinelli, F., Argente,
  E. (Eds.), Multi-Agent Systems and Agreement Technologies. Springer
  International Publishing, Cham, pp. 281--291.

\bibitem[{Hanawal et~al.(2019)Hanawal, Liu, Zhu, and
  Paschalidis}]{Han_lIU_zHU_Pas_19}
Hanawal, M., Liu, H., Zhu, H., Paschalidis, I., 2019. Learning policies for
  {M}arkov {D}ecision {P}rocesses from data. IEEE Transactions on Automatic
  Control 64, 2298--2309.

\bibitem[{Herzallah(2015)}]{Herzallah_R_JNeurNet_2015}
Herzallah, R., 2015. Fully probabilistic control for stochastic nonlinear
  control systems with input dependent noise. Neural networks 63, 199--207.

\bibitem[{Hiebert(1980)}]{Hiebert_K_SIAM_AppAn_1980_sol_eq&ineq}
Hiebert, K.~L., 1980. Solving systems of linear equations and inequalities.
  SIAM Journal on Numerical Analysis 17~(3), 447--464.

\bibitem[{Hou and Xu(2009)}]{Hou_09}
Hou, Z., Xu, J.-X., 2009. On data-driven control theory: the state of the art
  and perspective. Acta Automatica Sinica 35, 650--667.

\bibitem[{Hou and Wang(2013)}]{HOU20133}
Hou, Z.-S., Wang, Z., 2013. From model-based control to data-driven control:
  Survey, classification and perspective. Information Sciences 235, 3 -- 35.

\bibitem[{Kappen et~al.(2012)Kappen, Go\'mez, and Opper}]{Kappen_2012}
Kappen, H., Go\'mez, Opper, M., 2012. Optimal control as a graphical model
  inference problem. Machine Learning 87, 159--182.

\bibitem[{Karlin and Studden(1966)}]{10.2307/2238570}
Karlin, S., Studden, W.~J., 1966. Optimal experimental designs. The Annals of
  Mathematical Statistics 37~(4), 783--815.

\bibitem[{K{\'a}rn{\'y}(1996)}]{Karny_M_Automatica_1996_Towards_Fully_Prob}
K{\'a}rn{\'y}, M., 1996. Towards fully probabilistic control design. Automatica
  32~(12), 1719--1722.

\bibitem[{K{\'a}rn{\'y} and Guy(2006)}]{Karny_M+Guy_TV_Sys&Ctr_lett_2006}
K{\'a}rn{\'y}, M., Guy, T.~V., 2006. Fully probabilistic control design.
  Systems \& Control Letters 55~(4), 259--265.

\bibitem[{{Keel} and {Bhattacharyya}(2008)}]{4610025}
{Keel}, L.~H., {Bhattacharyya}, S.~P., 2008. Controller synthesis free of
  analytical models: Three term controllers. IEEE Transactions on Automatic
  Control 53~(6), 1353--1369.

\bibitem[{Kirk(2004)}]{Kirk_DE_2004_OptimControl}
Kirk, D.~E., 2004. Optimal control theory: an introduction. Courier
  Corporation.

\bibitem[{Kullback and Leibler(1951)}]{KL_51}
Kullback, S., Leibler, R., 1951. On information and sufficiency. Annals of
  Mathematical Statistics 22, 79--87.

\bibitem[{Kárný and Kroupa(2012)}]{KARNY2012105}
Kárný, M., Kroupa, T., 2012. Axiomatisation of fully probabilistic design.
  Information Sciences 186~(1), 105 -- 113.

\bibitem[{{Markovsky} and {Rapisarda}(2007)}]{7068299}
{Markovsky}, I., {Rapisarda}, P., 2007. On the linear quadratic data-driven
  control. In: 2007 European Control Conference (ECC). pp. 5313--5318.

\bibitem[{McKinnon and Schoellig(2019)}]{8651519}
McKinnon, C.~D., Schoellig, A.~P., 2019. Learn fast, forget slow: Safe
  predictive learning control for systems with unknown and changing dynamics
  performing repetitive tasks. IEEE Robotics and Automation Letters 4~(2),
  2180--2187.

\bibitem[{Moser et~al.(2017)Moser, Ramezani, Gagliardi, Zhou, and del
  Re}]{8287310}
Moser, D., Ramezani, Z., Gagliardi, D., Zhou, J., del Re, L., 2017. Risk
  functions oriented autonomous overtaking. In: 2017 11th Asian Control
  Conference (ASCC). pp. 1017--1022.

\bibitem[{Moser et~al.(2015)Moser, Waschl, Schmied, Efendic, and del
  Re}]{moser2015short}
Moser, D., Waschl, H., Schmied, R., Efendic, H., del Re, L., 2015. Short term
  prediction of a vehicle's velocity trajectory using {ITS}. SAE International
  Journal of Passenger Cars-Electronic and Electrical Systems 8~(2015-01-0295),
  364--370.

\bibitem[{Nakka et~al.(2021)Nakka, Liu, Shi, Anandkumar, Yue, and
  Chung}]{9290355}
Nakka, Y.~K., Liu, A., Shi, G., Anandkumar, A., Yue, Y., Chung, S.-J., 2021.
  Chance-constrained trajectory optimization for safe exploration and learning
  of nonlinear systems. IEEE Robotics and Automation Letters 6~(2), 389--396.

\bibitem[{Nguyen et~al.(2017)Nguyen, Moser, Schrangl, del Re, and
  Jones}]{8287308}
Nguyen, N.~A., Moser, D., Schrangl, P., del Re, L., Jones, S., 2017. Autonomous
  overtaking using stochastic model predictive control. In: 2017 11th Asian
  Control Conference (ASCC). pp. 1005--1010.

\bibitem[{{OpenStreetMap contributors}(2017)}]{OpenStreetMap2019}
{OpenStreetMap contributors}, 2017. {Planet dump retrieved from
  https://planet.osm.org }. \url{ https://www.openstreetmap.org }.

\bibitem[{Pavon and Ferrante(2006)}]{1618839}
Pavon, M., Ferrante, A., 2006. On the {G}eorgiou-{L}indquist approach to
  constrained {K}ullback-{L}eibler approximation of spectral densities. IEEE
  Transactions on Automatic Control 51~(4), 639--644.

\bibitem[{{Pegueroles} and {Russo}(2019)}]{Pegueroles_G+Russo_G_ECC19_confid}
{Pegueroles}, B.~G., {Russo}, G., June 2019. On robust stability of fully
  probabilistic control with respect to data-driven model uncertainties. In:
  2019 18th European Control Conference (ECC). pp. 2460--2465.

\bibitem[{Peterka(1981)}]{Peterka_V_Bayesian_Approach_to_sys_ident_1981}
Peterka, V., 1981. Bayesian approach to system identification. Elsevier, pp.
  239--304.

\bibitem[{Ramachandran and Amir(2007)}]{Ramachandran:2007:BIR:1625275.1625692}
Ramachandran, D., Amir, E., 2007. Bayesian inverse reinforcement learning. In:
  Proceedings of the 20th International Joint Conference on Artifical
  Intelligence. IJCAI'07. Morgan Kaufmann Publishers Inc., San Francisco, CA,
  USA, pp. 2586--2591.

\bibitem[{Ratliff et~al.(2006)Ratliff, Bagnell, and
  Zinkevich}]{Ratliff:2006:MMP:1143844.1143936}
Ratliff, N.~D., Bagnell, J.~A., Zinkevich, M.~A., 2006. Maximum margin
  planning. In: Proceedings of the 23rd International Conference on Machine
  Learning. ICML '06. ACM, New York, NY, USA, pp. 729--736.

\bibitem[{Ratliff et~al.(2009)Ratliff, Silver, and Bagnell}]{Ratliff2009}
Ratliff, N.~D., Silver, D., Bagnell, J.~A., Jul 2009. Learning to search:
  Functional gradient techniques for imitation learning. Autonomous Robots
  27~(1), 25--53.

\bibitem[{Rockafeller(1976)}]{Roc_76}
Rockafeller, R., 1976. Duality and stablity in extremum problems involving
  convex functions. Pacific Journal of Mathematics 21, 167--186.

\bibitem[{{Rosolia} and {Borrelli}(2018)}]{8039204}
{Rosolia}, U., {Borrelli}, F., 2018. Learning model predictive control for
  iterative tasks. {A} data-driven control framework. IEEE Transactions on
  Automatic Control 63~(7), 1883--1896.

\bibitem[{Russo(2021)}]{9244209}
Russo, G., 2021. On the crowdsourcing of behaviors for autonomous agents. IEEE
  Control Systems Letters 5~(4), 1321--1326.

\bibitem[{Salvador et~al.(2018)Salvador, delaPena, Alamo, and
  Bemporad}]{SALVADOR2018356}
Salvador, J.~R., delaPena, D.~M., Alamo, T., Bemporad, A., 2018. Data-based
  predictive control via direct weight optimization. IFAC-PapersOnLine 51~(20),
  356 -- 361, 6th IFAC Conference on Nonlinear Model Predictive Control NMPC
  2018.

\bibitem[{Singh and Vishnoi(2014)}]{10.1145/2591796.2591803}
Singh, M., Vishnoi, N.~K., 2014. Entropy, optimization and counting. In:
  Proceedings of the Forty-Sixth Annual ACM Symposium on Theory of Computing.
  STOC ’14. Association for Computing Machinery, New York, NY, USA, p.
  50–59.

\bibitem[{Tanaskovic et~al.(2017)Tanaskovic, Fagiano, Novara, and
  Morari}]{TANASKOVIC20171}
Tanaskovic, M., Fagiano, L., Novara, C., Morari, M., 2017. Data-driven control
  of nonlinear systems: An on-line direct approach. Automatica 75, 1 -- 10.

\bibitem[{Todorov(2007)}]{NIPS2006_d806ca13}
Todorov, E., 2007. Linearly-solvable {M}arkov decision problems. In:
  Sch\"{o}lkopf, B., Platt, J., Hoffman, T. (Eds.), Advances in Neural
  Information Processing Systems. Vol.~19. MIT Press.

\bibitem[{Todorov(2009)}]{Todorov11478}
Todorov, E., 2009. Efficient computation of optimal actions. Proceedings of the
  National Academy of Sciences 106~(28), 11478--11483.

\bibitem[{van Waarde(2021)}]{9406124}
van Waarde, H.~J., 2021. Beyond persistent excitation: Online experiment design
  for data-driven modeling and control. IEEE Control Systems Letters, 1--1.

\bibitem[{van Waarde et~al.(2020)van Waarde, De~Persis, Camlibel, and
  Tesi}]{9062331}
van Waarde, H.~J., De~Persis, C., Camlibel, M.~K., Tesi, P., 2020. Willems’
  fundamental lemma for state-space systems and its extension to multiple
  datasets. IEEE Control Systems Letters 4~(3), 602--607.

\bibitem[{{Van Waarde} et~al.(2020){Van Waarde}, {Eising}, {Trentelman}, and
  {Camlibel}}]{8960476}
{Van Waarde}, H.~J., {Eising}, J., {Trentelman}, H.~L., {Camlibel}, M.~K.,
  2020. Data informativity: a new perspective on data-driven analysis and
  control. IEEE Transactions on Automatic Control 65~(11), 4753--4768.

\bibitem[{Vitus and Tomlin(2013)}]{6760249}
Vitus, M.~P., Tomlin, C.~J., 2013. A probabilistic approach to planning and
  control in autonomous urban driving. In: 52nd IEEE Conference on Decision and
  Control. pp. 2459--2464.

\bibitem[{{Wabersich} and {Zeilinger}(2018)}]{8550288}
{Wabersich}, K.~P., {Zeilinger}, M.~N., June 2018. Scalable synthesis of safety
  certificates from data with application to learning-based control. In: 2018
  European Control Conference (ECC). pp. 1691--1697.

\bibitem[{{Xu} and {Paschalidis}(2019)}]{8796280}
{Xu}, T., {Paschalidis}, I.~C., June 2019. Learning models for writing better
  doctor prescriptions. In: 2019 18th European Control Conference (ECC). pp.
  2454--2459.

\bibitem[{Zhu and Baggio(2019)}]{8359301}
Zhu, B., Baggio, G., 2019. On the existence of a solution to a spectral
  estimation problem à la {B}yrnes–{G}eorgiou–{L}indquist. IEEE
  Transactions on Automatic Control 64~(2), 820--825.

\bibitem[{Zhu et~al.(2020)Zhu, Liu, Ataei, Munk, Daniel, and
  Paschalidis}]{10.1371/journal.pcbi.1007452}
Zhu, H., Liu, H., Ataei, A., Munk, Y., Daniel, T., Paschalidis, I.~C., 01 2020.
  Learning from animals: How to navigate complex terrains. PLOS Computational
  Biology 16~(1), 1--17.

\bibitem[{Ziebart et~al.(2008)Ziebart, Maas, Bagnell, and
  Dey}]{ziebart2008maximum}
Ziebart, B.~D., Maas, A., Bagnell, J.~A., Dey, A.~K., 2008. Maximum entropy
  inverse reinforcement learning. In: Proc. AAAI. pp. 1433--1438.

\bibitem[{Ziegler and Nichols(1942)}]{10.1115/1.2899060}
Ziegler, J.~G., Nichols, N.~B., 1942. {Optimum Settings for Automatic
  Controllers}. Transactions of the ASME 64, 759–768.

\end{thebibliography}
